\documentclass[12pt,notitlepage,a4paper,reqno,ps]{amsart}
\usepackage{amsfonts}
\usepackage{eurosym}
\usepackage{amssymb}
\usepackage{amstext}
\usepackage{color}

\definecolor{lightgray}{rgb}{0.83, 0.83, 0.83}

\newcommand{\medint}{-\kern  -,375cm\int}

\newenvironment{michelarev}{\color{red}}{\color{black}}
\newcommand{\bmicr}{\begin{michelarev}}
\newcommand{\emicr}{\end{michelarev}}
\theoremstyle{plain}
\newtheorem{theorem}{Theorem}[section]
\newtheorem{corollary}[theorem]{Corollary}
\newtheorem{lemma}[theorem]{Lemma}

\theoremstyle{definition}

\theoremstyle{remark}
\newtheorem{remark}[theorem]{Remark}
\newtheorem{example}[theorem]{Example}

\numberwithin{equation}{section} \makeatletter

\renewcommand{\p@enumi}{\thesection.}
\makeatother  \allowdisplaybreaks
\email{michela.eleuteri@unimore.it}
  \email{paolo.marcellini@unifi.it}
 \email{elvira.mascolo@unifi.it}
 \email{stefania.perrotta@unimore.it}

\begin{document}
\title[Local Lipschitz continuity with slow growth ]{Local Lipschitz
continuity for energy integrals \\
with slow growth }
\author[M. Eleuteri -- P. Marcellini -- E. Mascolo -- S. Perrotta]{Michela
Eleuteri -- Paolo Marcellini -- Elvira Mascolo -- Stefania Perrotta}
\address{Dipartimento di Scienze Fisiche, Informatiche e Matematiche,
Universit\`a degli Studi di Modena e Reggio Emilia, via Campi 213/b, 41125 -
Modena, Italy}
\address{Dipartimento di Matematica e Informatica ``U. Dini'', Universit\`a
di Firenze\\
Viale Morgagni 67/A, 50134 - Firenze, Italy}
\thanks{The authors are members of GNAMPA (Gruppo Nazionale per l'Analisi
Matematica, la Probabilit\`a e le loro Applica%
${{}^1}$%
zioni) of INdAM (Istituto Nazionale di Alta Matematica).}

\begin{abstract}
We consider some energy integrals under slow growth and we prove that the
local minimizers are locally Lipschitz continuous. Many examples are given,
either with subquadratic $p,q-$growth and/or anisotropic growth.
\end{abstract}

\maketitle

\begin{center}
\fbox{\today}
\end{center}

\section{Prologue\label{Section Prologue}}

When concerned with the $W^{1,\infty }$ or $C^{1,\alpha }$ regularity of
local minimizers of energy integrals of the calculus of variations of the
type 
\begin{equation}
F(u)=\int_{\Omega }f\left( Du\left( x\right) \right) \,dx\,
\label{energy integral}
\end{equation}%
we are naturally led to require a \textit{qualified convexity condition} on
the energy integrand $f:\mathbb{R}^{n}\rightarrow \mathbb{R}$; more
precisely, on the quadratic form of the $n\times n$ matrix of the second
derivatives $D^{2}f=\left( f_{\xi _{i}\xi _{j}}\right) $ of $f$ 
\begin{equation}
g_{1}\left( \left\vert \xi \right\vert \right) \left\vert \lambda
\right\vert ^{2}\leq \sum_{i,j=1}^{n}f_{\xi _{i}\xi _{j}}\left( \xi \right)
\lambda _{i}\lambda _{j}\leq g_{2}\left( \left\vert \xi \right\vert \right)
\left\vert \lambda \right\vert ^{2}\,,\;\;\;\forall \;\lambda ,\xi \in 
\mathbb{R}^{n},  \label{ellipticity condition in the prologue}
\end{equation}%
where $g_{1},g_{2}:\left[ 0,+\infty \right) \rightarrow \left[ 0,+\infty
\right) $ are given nonnegative real functions which allow us to control the 
\textit{ellipticity} in the minimization problem. Of course $g_{1}\left(
t\right) \leq g_{2}\left( t\right) $ for all $t\in \left[ 0,+\infty \right) $%
; if $g_{1}$ is positive and there exists a constant $M\geq 1$ such that $%
g_{2}\left( t\right) \leq Mg_{1}\left( t\right) $ for all $t\in \left[
0,+\infty \right) $ then we say that the we are dealing with an \textit{%
uniformly elliptic} problem. This is the case when the quadratic form $%
\sum_{i,j=1}^{n}f_{\xi _{i}\xi _{j}}\left( \xi \right) \lambda _{i}\lambda
_{j}$ has a simpler equivalent behavior as $g_{1}\left( \left\vert \xi
\right\vert \right) \left\vert \lambda \right\vert ^{2}$ and $g_{2}\left(
\left\vert \xi \right\vert \right) \left\vert \lambda \right\vert ^{2}$ and
the regularity process can work easier. However the assumption $g_{2}\leq
Mg_{1}$ rules out many interesting energy integrals; in this paper we do not
assume this uniformly elliptic condition.

For instance, in the special case $f\left( \xi \right) =g\left( \left\vert
\xi \right\vert \right) $ with $g:\left[ 0,+\infty \right) \rightarrow 
\mathbb{R}$, a direct computation (see for instance (6.3) in \cite{mar93}
and \cite{MP}) shows that 
\begin{equation*}
g_{1}\left( t\right) =\min \left\{ g^{\prime \prime }\left( t\right) ,\frac{%
g^{\prime }\left( t\right) }{t}\right\} ,\;\;\;\;g_{2}\left( t\right) =\max
\left\{ g^{\prime \prime }\left( t\right) ,\frac{g^{\prime }\left( t\right) 
}{t}\right\} ,\;\;\;\;\forall \;t\in \left[ 0,+\infty \right) ,
\end{equation*}%
where $g^{\prime },g^{\prime \prime }$ are the first and the second
derivatives of $g$. In this context the \textit{uniformly elliptic} case
corresponds to compare $g^{\prime \prime }\left( t\right) $ and $\frac{%
g^{\prime }\left( t\right) }{t}$; i.e. to ask for the existence of two
positive constants $m,M$ such that $m\frac{g^{\prime }\left( t\right) }{t}%
\leq g^{\prime \prime }\left( t\right) \leq M\frac{g^{\prime }\left(
t\right) }{t}$ for all $t\in \left( 0,+\infty \right) $. The $p-$Laplacian $%
f\left( \xi \right) =\left\vert \xi \right\vert ^{p}$ with $p>1$ is a main
example of \textit{uniformly elliptic energy integrand}, with $g\left(
t\right) =t^{p}$ and $\frac{g^{\prime }\left( t\right) }{t}=\frac{1}{p-1}%
g^{\prime \prime }\left( t\right) =pt^{p-2}$. Within this uniformly elliptic
context - however nonlinearities of possibly non-polynomial type are allowed
- we quote the global (i.e., up to the boundary) Lipschitz regularity
results by Cianchi-Maz'ya for a class of quasilinear elliptic equations \cite%
{Cianchi-Maz'ya2010} and for a class of nonlinear elliptic systems \cite%
{Cianchi-Maz'ya2014}.

Also some energy integrands of $p,q-$\textit{growth} can be uniformly
elliptic; for instance an integrand, which does not behave like a power, but
which however is an uniformly elliptic energy integrand, is $f(\xi
)=\left\vert \xi \right\vert ^{a+b\sin (\log \log |\xi |)}$; in this case $%
f(\xi )$ is a convex function for $\left\vert \xi \right\vert \geq e$ if $%
a,b>0$ and $a>1+b\sqrt{2}$. This function $f$ satisfies the $p,q-$\textit{%
growth conditions} with $p=a-b$ and $q=a+b$. It can be shown (see \cite%
{Boegelein-Duzaar-Marcellini-Scheven-2018},\cite%
{Boegelein-Duzaar-Marcellini-Scheven-2019}) that $f(\xi )$ satisfies the $%
\Delta _{2}$-condition. To notice however that some convex functions $%
f\left( \xi \right) =g\left( \left\vert \xi \right\vert \right) $ of $p,q-$%
growth with $p>1$ and $q>p$ arbitrarily close to $p$ exist, they do not
satisfy the $\Delta _{2}$-condition and the corresponding variational
problem are not uniformly elliptic; see Krasnosel'skij-Rutickii \cite[%
p.~28--29]{1961 Krasnosel'skij-Rutickii}, Focardi-Mascolo \cite[p.~342--343]%
{2001 Focardi-Mascolo}, Chlebicka \cite[Section~2.4]{Chlebicka} and B\"{o}%
gelein-Duzaar-Marcellini-Scheven \cite[Remark 3.3]%
{Boegelein-Duzaar-Marcellini-Scheven-2019}.

In this research we are concerned with the $W^{1,\infty }$ regularity of the
local minimizers of energy integrals of the calculus of variations of the
type (\ref{energy integral}), when the quadratic form of the second
derivatives $D^{2}f=\left( f_{\xi _{i}\xi _{j}}\right) $ of $f$ is governed
by (\ref{ellipticity condition in the prologue}) where $g_{1},g_{2}:\left[
0,+\infty \right) \rightarrow \left[ 0,+\infty \right) $ are given
nonnegative real functions, not only of polynomial type. The first local
Lipschitz-continuity results under this general context has been proposed in
the '90s in \cite{mar93},\cite{mar96} by assuming, among other conditions,
that $g_{1}\left( t\right) ,g_{2}\left( t\right) $ are increasing functions
in $\left[ 0,+\infty \right) $. This, when typified by the model case $%
g_{1}\left( t\right) =t^{p-2},g_{2}\left( t\right) =t^{q-2}$, gives $q\geq
p\geq 2$. The approach to regularity, governed by (\ref{ellipticity
condition in the prologue}) with general $g_{1},g_{2}$ functions not
necessarily monotone functions, was given by Marcellini-Papi \cite{MP}.
Related regularity results, with energy-integrands $f\left( x,\xi \right)
=g\left( x,\left\vert \xi \right\vert \right) $ depending on $x$ too, are
due to Mascolo-Migliorini \cite{Mascolo-Migliorini 2003}, Beck-Mingione \cite%
{Beck-Mingione 2020}, Di Marco-Marcellini \cite{DiMarco-Marcellini}, De
Filippis-Mingione \cite{DeFilippis-Mingione 2020}. See also
Apushkinskaya-Bildhauer-Fuchs \cite{Apushkinskaya-Bildhauer-Fuchs 2010} for
a local gradient bound of a-priori bounded minimizers.

More precisely, in \cite{Beck-Mingione 2020} Beck-Mingione consider the
vector-valued case of maps $u:\Omega \subset \mathbb{R}^{n}\rightarrow 
\mathbb{R}^{m}$ and $f\left( x,\xi \right) =g\left( \left\vert \xi
\right\vert \right) +h\left( x\right) u$ with the main part $g\left(
\left\vert \xi \right\vert \right) $ modulus dependent, as well as in \cite%
{Mascolo-Migliorini 2003},\cite{DiMarco-Marcellini},\cite%
{DeFilippis-Mingione 2020}; they also study the scalar case $m=1$ with the
more general integrand not modulus dependent, i.e. of the form $f\left(
x,\xi \right) =g\left( \xi \right) +h\left( x\right) u$, however with a
growth assumption from below of power type for some fixed exponent greater
than $1$ (see (1.33) in \cite{Beck-Mingione 2020}).

Here we focus our attention to \textit{slow growth} integrands, for which
the state-of-art is not so established. We give a general local $W^{1,\infty
}$ regularity result for the minimizers of energy integrals of the type (\ref%
{energy integral}),(\ref{ellipticity condition in the prologue}) with an
energy integrand $f=f\left( \xi \right) $ not necessarily depending on the
modulus of $\xi $ and with $g_{1},g_{2}:\left[ 0,+\infty \right) \rightarrow %
\left[ 0,+\infty \right) $ nonnegative decreasing real functions (more
precisely we require that only $g_{2}$ is a decreasing function), not only
of polynomial type. Precise statements can be found in the next section. We
treat \textit{general slow growth conditions} under the ellipticity
condition (\ref{ellipticity condition in the prologue}), where $g_{2}\left(
t\right) $ is a \textit{decreasing} (not necessarily strictly decreasing)
function with respect to $t$; of course in the model case $g_{2}\left(
t\right) =M\left( 1+\left\vert \xi \right\vert ^{2}\right) ^{\frac{q-2}{2}}$
this corresponds to $q\leq 2$. As already said, in this article we do not
assume uniformly elliptic conditions, nor the modulus dependence as $f\left(
\xi \right) =g\left( \left\vert \xi \right\vert \right) $.

In this regularity field specific references for \textit{slow growth} are
Fuchs-Mingione \cite{Fuchs-Mingione}, who concentrated on the \textit{%
nearly-linear growth}, such as for instance the \textit{logarithmic case} $%
f(\xi )=g\left( \left\vert \xi \right\vert \right) =|\xi |\log (1+|\xi |)$;
also Bildhauer, in his book \cite{Bildhauer 2003}, considered \textit{%
nearly-linear growth}. Leonetti-Mascolo-Siepe \cite{leo-mas-sie2} considered
the \textit{subquadratic }$p,q-$\textit{growth} with $1<p<q<2$, with energy
densities for instance the type $f\left( \xi \right) =g\left( \left\vert \xi
\right\vert \right) =|\xi |^{p}\log ^{\alpha }(1+|\xi |)$.

Here we emphasize some examples which enter in our regularity theory and
which seem not to be considered in the mathematical literature on this
subject. The first one is complementary to the case considered by
Bousquet-Brasco \cite{Bousquet-Brasco 2019} for exponents $p_{i}\geq 2$ for
all $i=1,2,\ldots ,n$; in fact here we can treat the model energy-integral
(see Example~\ref{orthotopic})%
\begin{equation}
F_{1}(u)=\int_{\Omega }\sum_{i=1}^{n}\left( 1+u_{x_{i}}^{2}\right) ^{\frac{%
p_{i}}{2}}\,dx\,  \label{integral with p_i}
\end{equation}%
when $1<p_{i}\leq 2$ for all $i=1,2,\ldots ,n$. In Section \ref{Section: New
examples of anisotropic energy functions} we propose some further examples
of anisotropic energy integrands which seem to be new in the mathematical
literature on this subject.

The mail regularity result that we propose in this manuscript is Theorem~\ref%
{superlinear} stated in the next section. It gives a more general regularity
result than similar results that can be found in the recent mathematical
literature on $p,q-$growth; see in particular the Remark \ref{Remark} for
details. Also the integral $\int_{\Omega }\left\vert Du\right\vert \log
^{a}\left( 1+\left\vert Du\right\vert \right) \,dx$, for every $a>0$, enters
in the regularity result of Theorem \ref{superlinear}. Of course a
by-product of our general Theorem \ref{superlinear} is also the $p,q-$growth
case, when the ellipticity conditions (\ref{ellipticity condition in the
prologue}) are satisfied with $g_{1}\left( \left\vert \xi \right\vert
\right) =m\,\left\vert \xi \right\vert ^{p-2}$, $g_{2}\left( \left\vert \xi
\right\vert \right) =M\left( 1+\left\vert \xi \right\vert ^{2}\right) ^{%
\frac{q-2}{2}}$, for some positive constants $m,M$ and exponents $1<p\leq
q\leq 2$ such that $\frac{q}{p}<1+\frac{2}{n}$\thinspace . As well known,
this condition guarantees the Lipschitz continuity of the solutions also
when $q\geq p>1$ and classically this is nowadays a well known constraint
for the $p,q-$growth (see \cite{mar91},\cite{mar96},\cite{Marcellini JMAA
2020}).

The regularity results are stated in the next section, while in Sections \ref%
{Examples} and \ref{Section: New examples of anisotropic energy functions}
some examples are considered in more details. The other sections are devoted
to the proofs.

\section{Introduction and statement of the main results\label{Section
Introduction}}

We assume that $f:\mathbb{R}^{n}\rightarrow \lbrack 0\,,+\infty )$ is a
convex function in $\mathcal{C}(\mathbb{R}^{n})\cap \mathcal{C}^{2}(\mathbb{R%
}^{n}\backslash B_{t_{0}}(0))$ for some $t_{0}\geq 0$, satisfying the
following growth condition: there exist two continuous functions $%
g_{1},g_{2}:[t_{0}\,,+\infty )\rightarrow (0\,,+\infty )$ and positive
constants $C_{1}$, $C_{2}$, $\alpha $, $\beta $ and $\mu\in[0\,,1] $ such
that 
\begin{equation}
\left\{ 
\begin{split}
& g_{1}\left( \left\vert \xi \right\vert \right) \left\vert \lambda
\right\vert ^{2}\leq \sum_{i,j=1}^{n}f_{\xi _{i}\xi _{j}}\left( \xi \right)
\lambda _{i}\lambda _{j}\leq g_{2}\left( \left\vert \xi \right\vert \right)
\left\vert \lambda \right\vert ^{2}\,,\;\;\;\forall \;\lambda ,\xi \in 
\mathbb{R}^{n},\;\left\vert \xi \right\vert {\geq }t_{0} \\
& t\mapsto t^{\mu}g_{2}(t)\text{ is decreasing and $t\mapsto tg_{2}(t)$ is
increasing} \\
&\left( g_{2}(t)\right) ^{\frac{n-2}{n}}\leq C_{1}t^{2\beta }g_{1}(t),\;\;\;%
\text{$\frac{1}{n}<\beta <\frac{2}{n}$},\;\forall \;\text{$t\geq t_{0}$} \\
& g_{2}(\left\vert \xi \right\vert )|\xi |^{2}\leq C_{2}\,[1+f(\xi
)]^{\alpha },\;\;\text{$\alpha >1$},\;\forall \text{ $\xi \in \mathbb{R}^{n}$%
},\;\left\vert \xi \right\vert \text{$\geq t_{0}$} \\
& f(\xi )/\left\vert \xi \right\vert \rightarrow +\infty \;\;\;\text{as }%
\left\vert \xi \right\vert \text{$\rightarrow \infty $}
\end{split}%
\right.  \label{H}
\end{equation}%
where $\frac{n-2}{n}$ in (\ref{H})$_{3}$, in the case $n=2$, must be
replaced with any fixed positive number less than ${1-\beta }$.

\bigskip

It is worth to highlight that we require uniform convexity and growth
assumptions on $f=f\left( \xi \right) $ only for large value of $\left\vert
\xi \right\vert $ (\cite{CE86},\cite{EMM1},\cite{EMM2},\cite{EMM3}). We say
that $u\in W_{loc}^{1,1}(\Omega )$ is a \textit{local minimizer} of the
integral functional $F$ in \eqref{energy integral} if $f\left( Du\right) \in
L_{loc}^{1}(\Omega )$ and 
\begin{equation*}
\int_{\Omega ^{\prime }}f(Du)\,dx\leq \int_{\Omega ^{\prime }}f(Du+D\varphi
)\,dx
\end{equation*}%
for every open set $\Omega ^{\prime }$, $\overline{\Omega ^{\prime }}\subset
\Omega $ and for every $\varphi \in W_{0}^{1,1}(\Omega ^{\prime })$. The
result for \textit{slow growth conditions}, under the ellipticity condition (%
\ref{ellipticity condition in the prologue}) with $g_{1}\left( t\right) $
and $g_{2}\left( t\right) $ general functions, can be stated as follows.

\begin{theorem}[general growth]
\label{superlinear} Let us assume that $f$ satisfies the growth assumptions
in (\ref{H}), with the parameters $\alpha $, $\beta $, $\mu $ related by the
condition 
\begin{equation}
{2-\mu -\alpha (n\beta -\mu )}>0\,.  \label{ab}
\end{equation}%
Then any minimizer $u\in W_{\mathrm{loc}}^{1,1}(\Omega )$ of 
\eqref{energy
integral} is locally Lipschitz continuous in $\Omega $ and, for every $%
0<\rho <R$, $\bar{B}_{R}\subset \Omega $, there exists a positive constant $%
C $ depending on $\rho $, $R$, $C_{1}$, $C_{2}$, $\alpha $, $\beta $, $\mu $%
, $g_{2}(t_{0})$, such that 
\begin{equation}
\Vert Du\Vert _{L^{\infty }(B_{\rho }\,;\mathbb{R}^{n})}\leq \,C\,\left\{ 
\frac{1}{(R-\rho )^{n}}\int_{B_{R}}\{1+f(Du)\}\,dx\right\} ^{\theta }
\label{infinity nu}
\end{equation}%
where $\theta =\frac{(2-\mu )\alpha }{2-\mu -\alpha (n\beta -\mu )}$.
\end{theorem}

\bigskip

When we specialize Theorem \ref{superlinear} to the \textit{subquadratic }$%
p,q-$\textit{growth} we obtain:

\begin{corollary}[$p,q-$growth]
\label{theorem on p,q} Let $f=f\left( \xi \right) $ be a convex function in $%
\mathcal{C}(\mathbb{R}^{n})\cap \mathcal{C}^{2}(\mathbb{R}^{n}\backslash
B_{t_{0}}(0))$ for some $t_{0}\geq 0$, satisfying the ellipticity conditions 
\begin{equation}
m\,|\xi |^{p-2}|\lambda |^{2}\leq \sum_{i,j=1}^{n}f_{\xi _{i}\xi _{j}}(\xi
)\lambda _{i}\lambda _{j}\leq M\left( 1+\left\vert \xi \right\vert
^{2}\right) ^{\frac{q-2}{2}}|\lambda |^{2},\quad \forall \;\lambda ,\xi \in 
\mathbb{R}^{n}:\left\vert \xi \right\vert \geq t_{0}\,,
\label{ellipticity conditions under p,q}
\end{equation}%
for some positive constants $m,M$ and exponents $p,q$, $1\leq p\leq q\leq 2$%
, such that 
\begin{equation}
\frac{q}{p}<1+\frac{2}{n}.  \label{q/p for p,q<=2}
\end{equation}%
Then every local minimizer $u\in W_{\mathrm{loc}}^{1,p}\left( \Omega \right) 
$ to the energy integral in \eqref{energy integral} \textit{is of class }$W_{%
\mathrm{loc}}^{1,\infty }\left( \Omega \right) $\textit{\ and there exists a
constant }$C>0$, depending only on $p,q,n,m,M$, such that, for \textit{all }$%
\rho ,R$\textit{\ with }$0<\rho <R\leq \rho +1$, 
\begin{equation}
\left\Vert Du\right\Vert _{L^{\infty }\left( B_{\rho };\mathbb{R}^{n}\right)
}\leq C\left\{ \frac{1}{\left( R-\rho \right) ^{n}}\int_{B_{R}}\left\{
1+f\left( Du\right) \right\} \,dx\right\} ^{\frac{2}{(n+2)p-nq}}.
\label{bound for p,q<=2}
\end{equation}
\end{corollary}

\bigskip

Let us briefly sketch the tools and the techniques to prove the above
regularity results. A first step is an \textit{a priori estimate} for smooth
minimizers through the interpolation result stated in Lemma \ref%
{interpolation}. The second step is an approximation procedure: we construct
a sequence of smooth strictly convex functions $f_{k}$, each of them being
equal to $f$ for large $\left\vert \xi \right\vert $, in the same outlook in 
\cite{MS},\cite{mar96}. More in details, if $u$ is a local minimizer of %
\eqref{energy integral}, we consider the sequence of variational problems in
a ball $B_{R}$, $\overline{B}_{R}\subset \Omega $, with as integrand a
suitable perturbation of $f_{k}$ and boundary value data $u_{\epsilon
}=u\ast \varphi _{\varepsilon }$, where $\varphi _{\varepsilon }$ are smooth
mollifiers. Each $u_{\epsilon }$ satisfies the \textit{bounded slope
condition}; then, by the well know existence and Lipschitz regularity
theorem by Hartman-Stampacchia \cite{Hartman-Stampacchia} each problem has a
unique Lipschitz continuous solution $v_{\epsilon }$. By applying the
a-priori estimate to the sequence of the solutions we get an uniform control
in $L^{\infty }$ of the gradient of $v_{\epsilon }$, which allows us to
transfer the Lipschitz continuity property to the original minimizer $u$.

\smallskip The plan of the paper is the following: in Sections \ref{Examples}%
, \ref{Section: New examples of anisotropic energy functions} we present
some examples, some of them being new in this context of general growth
conditions. In Section \ref{Preliminary lemma} we give the \textit{%
interpolation lemma}. In Section \ref{A priori estimates} we prove the a
priori estimate for functionals with \textit{general slow growth} by means
of the interpolation lemma. In the last section we prove the regularity
results. As we show in the next section, the class of energy integrals that
we consider is quite large, not only polynomial unbalanced $p,q-$\textit{%
subquadratic growth} as in the Corollary \ref{theorem on p,q}, but also 
\textit{logaritmic growth} (as in Examples \ref{log} and \ref{loglog}) and 
\textit{anisotropic behaviour} (Example \ref{orthotopic}).

\bigskip

\section{Examples\label{Examples}}

In this section we present some examples of density function $f$ for which
the above assumptions hold.

\smallskip

\begin{example}
\label{log} $f(\xi)=|\xi|(\log|\xi|)^a$, $a>0$, $|\xi|\geq t_0\geq 1$. For
large $t$ \eqref{H}$_1$ holds for $g_1(t)=\frac{a}2\frac{(\log t)^{a-1}}{t}$
and $g_2(t)=(1+a)\frac{(\log t)^a}{t}$. It is easy to check that \eqref{H}$%
_2 $ and \eqref{H}$_3$ hold for every $\beta>\frac 1n$. Since $%
g_2(|\xi|)|\xi|^2= (1+a)f(\xi)$, \eqref{H}$_4$ holds for every $\alpha>1$.
Moreover, for every $\mu<1$, $t^{\mu}g_2(t)$ is decreasing in $%
[t_0\,,+\infty)$, choosing $\alpha>\frac{1}{n\beta-1}$, \eqref{ab} follows.
Therefore Theorem \ref{superlinear} applies for every $a>0$.
\end{example}

\smallskip

\begin{example}
\label{loglog} $f(\xi )=(|\xi |+1)L_{k}(|\xi |)$, $g(t)=(1+t)L_{k}(t)$, $%
k\in \mathbb{N}$, $L_{k}$ defined as: 
\begin{equation*}
L_{1}\left( t\right) =\log \left( 1+t\right) ,\;\;\;\;\;L_{k+1}\left(
t\right) =\log \left( 1+L_{k}\left( t\right) \right) ;
\end{equation*}%
therefore 
\begin{equation*}
L_{1}^{\prime }\left( t\right) =\frac{1}{1+t}\,,\;\;\;\;L_{k+1}^{\prime }(t)=%
\frac{L_{k}^{\prime }(t)}{1+L_{k}(t)}=\frac{1}{(1+t)(1+L_{1}(t))\cdots
(1+L_{k-1}(t))}\,.
\end{equation*}%
Then we get 
\begin{equation*}
g^{\prime }(t)=L_{k}(t)+\frac{1}{(1+L_{1}(t))\cdots (1+L_{k-1}(t))}\quad
\Longrightarrow \quad g_{2}(t)=\frac{2}{1+t}L_{k}(t)\,;
\end{equation*}%
\begin{equation*}
\begin{split}
g^{\prime \prime }(t)=\frac{1}{(1+t)(1+L_{1}(t))\cdots (1+L_{k-1}(t))}& %
\left[ 1-\sum_{i=1}^{k-1}\frac{1}{(1+L_{1}(t))\cdots (1+L_{i}(t))}\right] \\
\quad \Longrightarrow \quad g_{1}(t)=& \frac{1}{2(1+t)(1+L_{1}(t))\cdots
(1+L_{k-1}(t))}\,.
\end{split}%
\end{equation*}

Similarly to the Example \ref{log}, for $t$ large enough, $\mu =1$ and $%
\beta >\frac{1}{n}$, \eqref{H}$_{2}$ and \eqref{H}$_{3}$ hold. Moreover, $%
g_{2}(|\xi |)|\xi |^{2}\leq 2f(\xi )$; therefore \eqref{H}$_{4}$ holds for
every $\alpha >1$. Since we can choose $\alpha $ and $\beta $ such that %
\eqref{ab} holds, every local minimizer of the corresponding integral is
locally Lipschitz continuous (see \cite{Fuchs-Mingione} for related results).
\end{example}

\smallskip

\begin{example}
\label{orthotopic} Next we consider the \textit{anisotropic case} of the
energy integral in (\ref{energy integral}) with 
\begin{equation}
f\left( \xi \right) =\sum_{i=1}^{n}\left\vert \xi _{i}\right\vert
^{p_{i}},\;\;\;\;\;\xi =\left( \xi _{1},\xi _{2},\ldots ,\xi _{n}\right) \in 
\mathbb{R}^{n},  \label{p_i}
\end{equation}%
where the exponents $p_{i}$ are greater than or equal to {$2$} for all $%
i=1,2,\ldots ,n$. Of course $f\left( \xi \right) $ in (\ref{p_i}) is a
convex function in $\mathbb{R}^{n}.$ Note that the $n\times n$ matrix of the
second derivatives $D^{2}f=\left( f_{\xi _{i}\xi _{j}}\right) $ of $f$ is
diagonal and the corresponding quadratic form is given by 
\begin{equation}
\sum_{i,j=1}^{n}f_{\xi _{i}\xi _{j}}\left( \xi \right) \lambda _{i}\lambda
_{j}=\sum_{i=1}^{n}p_{i}\left( p_{i}-1\right) \left\vert \xi _{i}\right\vert
^{p_{i}-2}\left\vert \lambda _{i}\right\vert ^{2}\,,\;\;\;\forall \;\lambda
,\xi \in \mathbb{R}^{n}.  \label{quadratic form}
\end{equation}%
This quadratic form is positive semidefinite but is not definite if (at
least) one of the exponents $p_{i}$ is greater than {$2$}; in fact, if for
instance $p_{1}>2$, then $\sum_{i,j=1}^{n}f_{\xi _{i}\xi _{j}}\left( \xi
\right) \lambda _{i}\lambda _{j}=0$ when $\xi =\left( \xi _{1},0,\ldots
,0\right) \neq 0$ and $\lambda =\left( 0,\lambda _{2},\ldots ,\lambda
_{n}\right) \neq 0$. Nevertheless, in spite of this lack of uniform
convexity, without using the quadratic form in (\ref{quadratic form}), the
local $L^{\infty }-$bound of the minimizers has been established in \cite%
{Cupini-Marcellini-Mascolo 2009}-\cite{Cupini-Marcellini-Mascolo 2018},\cite%
{Dibenedetto-Gianazza-Vespri},\cite{Fusco-Sbordone} under some optimal
conditions on the exponents $p_{i}>1$. More recently Bousquet-Brasco \cite%
{Bousquet-Brasco 2019} proved that \textit{bounded} minimizers of the energy
integral (\ref{energy integral}), with $f$ as in (\ref{p_i}), are locally
Lipschitz continuous in $\Omega $ under the condition $p_{i}\geq 2$ for all $%
i=1,2,\ldots ,n$.

In our context of slow growth we emphasize the locally Lipschitz regularity
that we deduce by Theorem \ref{superlinear} when $1<p_{i}\leq 2$ for all $%
i=1,2,\ldots ,n$, which should make more complete the case considered by
Bousquet-Brasco \cite{Bousquet-Brasco 2019}. More precisely, we have to
change the model example $f\left( \xi \right) $ in (\ref{p_i}) since $f:%
\mathbb{R}^{n}\rightarrow \mathbb{R}$ there is not a function of class $%
C^{2} $ around $\xi =0$ when $p_{i}<2$ for some $i\in \left\{ 1,2,\ldots
,n\right\} $. The corresponding not-singular model is%
\begin{equation}
f\left( \xi \right) =\sum_{i=1}^{n}\left( 1+\xi _{i}^{2}\right) ^{\frac{p_{i}%
}{2}},\;\;\;\;\;\xi =\left( \xi _{1},\xi _{2},\ldots ,\xi _{n}\right) \in 
\mathbb{R}^{n}.  \label{not singular p_i}
\end{equation}%
Similarly to (\ref{quadratic form}) we obtain the quadratic form of the $%
n\times n$ matrix of the second derivatives $D^{2}f=\left( f_{\xi _{i}\xi
_{j}}\right) $ of $f$ in (\ref{not singular p_i}) 
\begin{equation}
\sum_{i,j=1}^{n}f_{\xi _{i}\xi _{j}}\left( \xi \right) \lambda _{i}\lambda
_{j}=\sum_{i=1}^{n}p_{i}\left( 1+\left( p_{i}-1\right) \xi _{i}^{2}\right)
\left( 1+\xi _{i}^{2}\right) ^{\frac{p_{i}}{2}-2}\left\vert \lambda
_{i}\right\vert ^{2}\,,\;\;\;\forall \;\lambda ,\xi \in \mathbb{R}^{n}.
\label{not singular quadratic form}
\end{equation}%
Since $p_{i}-2\leq 0$ for every $i\in \left\{ 1,2,\ldots ,n\right\} $, then 
\begin{equation*}
\begin{split}
\left( 1+\left( p_{i}-1\right) \xi _{i}^{2}\right) \left( 1+\xi
_{i}^{2}\right) ^{\frac{p_{i}}{2}-2}& \geq (p_{i}-1)\left( 1+\xi
_{i}^{2}\right) ^{\frac{p_{i}-2}{2}} \\
& \geq (p_{i}-1)\left( 1+\left\vert \xi \right\vert ^{2}\right) ^{\frac{%
p_{i}-2}{2}}\geq (p-1)\left( 1+\left\vert \xi \right\vert ^{2}\right) ^{%
\frac{p-2}{2}}
\end{split}%
\end{equation*}%
where $p=:\min \left\{ p_{i}:i=1,2,\ldots ,n\right\} .$We obtain 
\begin{equation}
\sum_{i,j=1}^{n}f_{\xi _{i}\xi _{j}}\left( \xi \right) \lambda _{i}\lambda
_{j}\geq p\left( p-1\right) \left( 1+\left\vert \xi \right\vert ^{2}\right)
^{\frac{p-2}{2}}\left\vert \lambda \right\vert ^{2}\,,\;\;\;\forall
\;\lambda ,\xi \in \mathbb{R}^{n}.
\label{ellipticity condition in the example}
\end{equation}%
Again for every $i\in \left\{ 1,2,\ldots ,n\right\} $, since $p_{i}-2\leq 0$
we also have 
\begin{equation*}
p_{i}\left( 1+\left( p_{i}-1\right) \xi _{i}^{2}\right) \left( 1+\xi
_{i}^{2}\right) ^{\frac{p_{i}}{2}-2}\leq p_{i}\left( 1+\xi _{i}^{2}\right) ^{%
\frac{p_{i}-2}{2}}\leq p_{i}
\end{equation*}%
and thus from (\ref{not singular quadratic form}) we deduce 
\begin{equation}
\sum_{i,j=1}^{n}f_{\xi _{i}\xi _{j}}\left( \xi \right) \lambda _{i}\lambda
_{j}\leq 2\left\vert \lambda \right\vert ^{2}\,,\;\;\;\forall \;\lambda ,\xi
\in \mathbb{R}^{n}.  \label{growth condition in the example}
\end{equation}%
Therefore, by (\ref{ellipticity condition in the example}) and (\ref{growth
condition in the example}), we have 
\begin{equation}
g_{1}\left( \left\vert \xi \right\vert \right) \left\vert \lambda
\right\vert ^{2}\leq \sum_{i,j=1}^{n}f_{\xi _{i}\xi _{j}}\left( \xi \right)
\lambda _{i}\lambda _{j}\leq g_{2}\left( \left\vert \xi \right\vert \right)
\left\vert \lambda \right\vert ^{2}\,,\;\;\;\forall \;\lambda ,\xi \in 
\mathbb{R}^{n},  \label{final ellipticity conditions in the example}
\end{equation}%
where $g_{1},g_{2}:\left[ 0,+\infty \right) \rightarrow \left( 0,+\infty
\right) $ are the nonnegative real functions defined by $g_{1}\left(
t\right) =p\left( p-1\right) \left( 1+t^{2}\right) ^{\frac{p-2}{2}}$ and $%
g_{2}\left( t\right) =g_{2}$ constantly equal to $2$. By Corollary \ref%
{theorem on p,q} with $q=2$ we obtain the further regularity result too.

\begin{corollary}[anisotropic energy integrals with slow growth]
Let $f=f\left( \xi \right) $ be the model convex function in 
\eqref{not
singular p_i}, with $1<p_{i}\leq 2$ for all $i=1,2,\ldots ,n$. If 
\begin{equation}
\frac{2}{p}<1+\frac{2}{n}\text{\ \ }\Leftrightarrow \text{\ \ }p>\frac{2n}{%
n+2},\text{\ \ \ where\ \ \ \ }p=:\min_{i\in \left\{ 1,2,\ldots ,n\right\}
}\left\{ p_{i}\right\} ,  \label{bound for the p_i}
\end{equation}%
then every local minimizer $u\in W_{\mathrm{loc}}^{1,p}\left( \Omega \right) 
$ to the energy integral \eqref{energy integral}, with $f\left( \xi \right) $
in \eqref{not singular p_i}, \textit{is of class }$W_{\mathrm{loc}%
}^{1,\infty }\left( \Omega \right) $\textit{\ and there exists a constant }$%
C>0$ depending only on $p,n,m,M$, such that, for \textit{all }$\rho ,R$%
\textit{\ with }$0<\rho <R\leq \rho +1$, 
\begin{equation*}
\left\Vert Du\left( x\right) \right\Vert _{L^{\infty }\left( B_{\rho };%
\mathbb{R}^{n}\right) }\leq \left( \frac{C}{\left( R-\rho \right) ^{n}}%
\int_{B_{R}}\left\{ 1+f\left( Du\right) \right\} \,dx\right) ^{\frac{2}{%
(n+2)p-2n}}.
\end{equation*}
\end{corollary}

Note that when $n=2$ the bound in (\ref{bound for the p_i}) simply reduces
to $1<p_{i}\leq 2$ for all $i=1,2,\ldots ,n$. More generally we can consider
energy integrands of the form 
\begin{equation}
f\left( \xi \right) =\sum_{i=1}^{n}g\left( \xi _{i}\right) ,\;\;\;\;\text{or}%
\;\;\;\;f\left( \xi \right) =\sum_{i=1}^{n}g_{i}\left( \xi _{i}\right) ,
\label{general integrants}
\end{equation}%
where, for instance, $g\left( t\right) $ or $g_{i}\left( t\right) $ are one
of the functions considered above in Examples \ref{log} and \ref{loglog}.
\end{example}

\section{New examples of anisotropic energy functions\label{Section: New
examples of anisotropic energy functions}}

\smallskip We provide some applications of our Theorem \ref{superlinear} and
we infer the Lipschitz continuity of the local minimizers to some class of
functionals with anisotropic behaviour.

\begin{example}
\label{p-q-r-s} Consider 
\begin{equation}
f\left( \xi \right) =\sqrt{\sum_{i=1}^{n}\left( 1+\left\vert \xi
_{i}\right\vert ^{2}\right) ^{p_{i}}},\qquad p_{i}>1,\;\;\;\forall
\;i=1,\dots ,n.  \label{r-s}
\end{equation}%
With the same argument of Example \ref{orthotopic} we have 
\begin{equation*}
\frac{1}{\sqrt{n^{p}}}\,\left( 1+|\xi |^{2}\right) ^{\frac{p}{2}}\leq f(\xi
)\leq \sqrt{n}\,\left( 1+|\xi |^{2}\right) ^{\frac{q}{2}}
\end{equation*}%
where $p=\min_{i}p_{i}$ and $q=\max_{i}p_{i}$. Let us denote by $Q\left( \xi
,\lambda \right) $ the quadratic form 
\begin{equation}
Q\left( \xi ,\lambda \right) =\sum_{i,j=1}^{n}f_{\xi _{i}\xi _{i}}\left( \xi
\right) \lambda _{i}\lambda _{j}\,,\;\;\;\;\forall \;\lambda ,\xi \in 
\mathbb{R}^{n}.  \label{Q}
\end{equation}%
We have 
\begin{equation*}
\begin{split}
f_{\xi _{i}\xi _{i}}& =-\frac{p_{i}^{2}\xi _{i}^{2}(1+\xi
_{i}^{2})^{2p_{i}-2}}{\left( \sum_{k=1}^{n}(1+\xi _{k}^{2})^{p_{k}}\right) ^{%
\frac{3}{2}}}+\frac{p_{i}(1+\xi _{i}^{2})^{p_{i}-2}\left( 1+(2p_{i}-1)\xi
_{i}^{2}\right) }{\left( \sum_{i=1}^{n}(1+\xi _{i}^{2})^{p_{i}}\right) ^{%
\frac{1}{2}}},\qquad i=1,\dots n, \\
f_{\xi _{i}\xi _{j}}& =-\frac{p_{i}p_{j}\xi _{i}\xi _{j}(1+\xi
_{i}^{2})^{p_{i}-1}(1+\xi _{j}^{2})^{p_{j}-1}}{\left( \sum_{k=1}^{n}(1+\xi
_{k}^{2})^{p_{k}}\right) ^{\frac{3}{2}}},\qquad i,j=1,\dots n,\;\;i\neq j,
\end{split}%
\end{equation*}%
and then 
\begin{equation*}
\begin{split}
Q\left( \xi ,\lambda \right) \left( \sum_{k=1}^{n}(1+\xi
_{k}^{2})^{p_{k}}\right) ^{\frac{3}{2}}=& -(v\cdot w)^{2} \\
& +\left( \sum_{k=1}^{n}(1+\xi _{k}^{2})^{p_{k}}\right)
\sum_{i=1}^{n}p_{i}(1+\xi _{i}^{2})^{p_{i}-2}\left( 1+(2p_{i}-1)\xi
_{i}^{2}\right) \lambda _{i}^{2},
\end{split}%
\end{equation*}%
where $v_{i}=p_{i}\xi _{i}(1+\xi _{i}^{2})^{\frac{p_{i}}{2}-1}\lambda _{i}$
and $w_{i}=(1+\xi _{i}^{2})^{\frac{p_{i}}{2}}$. Therefore 
\begin{equation*}
\begin{split}
Q\left( \xi ,\lambda \right) & \left( \sum_{k=1}^{n}(1+\xi
_{k}^{2})^{p_{k}}\right) ^{\frac{1}{2}}\leq \sum_{i=1}^{n}p_{i}(1+\xi
_{i}^{2})^{p_{i}-2}\left( 1+(2p_{i}-1)\xi _{i}^{2}\right) \lambda _{i}^{2} \\
& \leq (2q^{2}-q)\sum_{i=1}^{n}\left[ (1+\xi _{i}^{2})^{p_{i}}\right] ^{1-%
\frac{1}{p_{i}}}\lambda _{i}^{2}\leq 2q^{2}\left[ \sum_{k=1}^{n}(1+\xi
_{k}^{2})^{p_{k}}\right] ^{1-\frac{1}{q}}|\lambda |^{2}.
\end{split}%
\end{equation*}%
For $|\xi |\geq 1$, if $q\leq 2$ we have 
\begin{equation}
\begin{split}
Q\left( \xi ,\lambda \right) & \leq 2q^{2}\left( \sum_{k=1}^{n}(1+\xi
_{k}^{2})^{p_{k}}\right) ^{\frac{q-2}{2q}}|\lambda |^{2}\leq 2q^{2}\left( 1+%
\frac{1}{n}|\xi |^{2}\right) ^{p\frac{q-2}{2q}}|\lambda |^{2} \\
& \leq C|\xi |^{\frac{p}{q}(q-2)}|\lambda |^{2},
\end{split}
\label{upper bound}
\end{equation}%
instead, if $q\geq 2$ we obtain 
\begin{equation*}
Q\left( \xi ,\lambda \right) \leq 2q^{2}\left( \sum_{k=1}^{n}(1+\xi
_{k}^{2})^{p_{k}}\right) ^{\frac{q-2}{2q}}|\lambda |^{2}\leq 2q^{2}\left(
n\left( 1+|\xi |^{2}\right) ^{q}\right) ^{\frac{q-2}{2q}}|\lambda |^{2}\leq
C|\xi |^{q-2}|\lambda |^{2}.
\end{equation*}%
Moreover, since 
\begin{equation*}
(v\cdot w)^{2}\leq |v|^{2}|w|^{2}=\sum_{i=1}^{n}p_{i}^{2}\xi _{i}^{2}(1+\xi
_{i}^{2})^{p_{i}-2}\lambda _{i}^{2}\sum_{k=1}^{n}(1+\xi _{k}^{2})^{p_{k}}
\end{equation*}%
we have 
\begin{equation*}
\begin{split}
Q\left( \xi ,\lambda \right) & \left( \sum_{k=1}^{n}(1+\xi
_{k}^{2})^{p_{k}}\right) ^{\frac{1}{2}}\geq -\sum_{i=1}^{n}p_{i}^{2}\xi
_{i}^{2}(1+\xi _{i}^{2})^{p_{i}-2}\lambda _{i}^{2} \\
& +\sum_{i=1}^{n}p_{i}(1+\xi _{i}^{2})^{p_{i}-2}\left( 1+(2p_{i}-1)\xi
_{i}^{2}\right) \lambda _{i}^{2}=\sum_{i=1}^{n}p_{i}(1+\xi
_{i}^{2})^{p_{i}-2}\left( 1+(p_{i}-1)\xi _{i}^{2}\right) \lambda _{i}^{2}.
\end{split}%
\end{equation*}%
For every $q>1$ and $|\xi |\geq 1$ we deduce 
\begin{equation}
\begin{split}
Q\left( \xi ,\lambda \right) & \geq \left( \sum_{k=1}^{n}(1+\xi
_{k}^{2})^{p_{k}}\right) ^{-\frac{1}{2}}(p^{2}-p)\sum_{i=1}^{n}(1+\xi
_{i}^{2})^{p_{i}-1}\lambda _{i}^{2} \\
& \geq \frac{p^{2}-p}{\sqrt{n}}(1+\max_{i}\{|\xi _{i}|\}^{2})^{p-1-\frac{q}{2%
}}|\lambda |^{2}\geq c|\xi |^{2p-2-q}|\lambda |^{2}.
\end{split}
\label{lower bound}
\end{equation}%
We note explicitly that if $1<p<q$ then $2p-2-q<p-2$. Therefore, by denoting 
\begin{equation}
r=2p-q\qquad \text{ and }\qquad s=\frac{p}{q}(q-2)+2  \label{r and s}
\end{equation}%
with $r\leq p\leq q\leq s\leq 2$, by \eqref{upper bound} and 
\eqref{lower
bound} we obtain that $f\left( \xi \right) $ in \eqref{p-q-r-s} satisfies
the assumptions \eqref{H}$_{1}$ and \eqref{H}$_{2}$ with 
\begin{equation}
g_{1}(t)=c\,t^{r-2}\qquad \text{ and }\qquad g_{2}(t)=C\,t^{s-2}.
\label{g1g2 in the example}
\end{equation}%
Therefore the function $f\left( \xi \right) $ in \eqref{p-q-r-s} satisfies
all assumptions in \eqref{H} with 
\begin{equation*}
\mu ={2-s},\qquad \beta =\frac{n-2}{2n}\,s-\frac{r}{2}+\frac{2}{n}\qquad 
\text{ and }\qquad \alpha =\frac{s}{p},
\end{equation*}%
when we impose the bounds 
\begin{equation*}
\alpha <\frac{2-\mu }{n\beta -\mu }\quad \iff \quad s<\frac{2}{n}\,p+r.
\end{equation*}%
We are in the conditions to apply Theorem~\ref{superlinear}. In the next
Corollary \ref{ex1-sec4} we state what we have proved by the computations
above for this example, about the energy integral 
\begin{equation}
F_{2}\left( u\right) =\int_{\Omega }{\left( \sum_{i=1}^{n}\left(
1+\left\vert u_{x_{i}}\right\vert ^{2}\right) ^{p_{i}}\right) ^{\frac{1}{2}}}%
dx\,.  \label{F2}
\end{equation}
\end{example}

\begin{corollary}
\label{ex1-sec4} Let $1<p=\min_{i}p_{i}\leq q=\max_{i}p_{i}\leq 2$ and $r,s$
as in (\ref{r and s}). If 
\begin{equation}
s<\frac{2}{n}\,p+r\quad \iff \quad \frac{q}{p}<1+\frac{2}{n}-2\left( \frac{1%
}{p}-\frac{1}{q}\right) \,,  \label{condition ex1}
\end{equation}%
then the local minimizers of the energy integral $F_{2}$ in \eqref{F2} are
locally Lipschitz continuous in $\Omega $.
\end{corollary}

\begin{remark}
\label{Remark}At a first glance we may think that the assumption (\ref%
{condition ex1}) of Corollary \ref{ex1-sec4} is more restrictive that the
similar assumption $\frac{q}{p}<1+\frac{2}{n}$ in Corollary \ref{theorem on
p,q}, valid under the general $p,q-$growth. But, if we apply correctly
Corollary \ref{theorem on p,q} to $F_{2}$, on the contrary we had a more
restrictive assumption than the above condition (\ref{condition ex1}). In
fact by (\ref{g1g2 in the example}) we have here $g_{1}\left( t\right)
=c\,t^{r-2}$, $g_{2}\left( t\right) =C\,t^{s-2}$ and the estimate of the
quadratic form \eqref{H}$_{1}$ becomes 
\begin{equation}
c\,\left\vert \xi \right\vert ^{r-2}\left\vert \lambda \right\vert ^{2}\leq
\sum_{i,j=1}^{n}f_{\xi _{i}\xi _{j}}\left( \xi \right) \lambda _{i}\lambda
_{j}\leq C\left\vert \xi \right\vert ^{s-2}\left\vert \lambda \right\vert
^{2},\quad \forall \;\lambda ,\xi \in 
\mathbb{R}^{n}:\left\vert \xi \right\vert \geq 1\,.  
\label{g1g2}
\end{equation}%
Then Corollary \ref{theorem on p,q} applied to $F_{2}$ gives the regularity
of minimizers under the bound $s<\left( 1+\frac{2}{n}\right) r=\frac{2}{n}%
\,r+r$, which is a more restrictive condition than the above assumption (\ref%
{condition ex1}) $s<\frac{2}{n}\,p+r$, since $r=2p-q=p+\left( p-q\right) <p$
when $p<q$.

Therefore the general bound of Corollary \ref{theorem on p,q} gives a less
precise result than Theorem~\ref{superlinear} when applies to the energy
integral (\ref{F2}). This fact also shows that Theorem~\ref{superlinear}
gives a more general regularity result than similar results that can be
found in the recent mathematical literature on $p,q-$growth.
\end{remark}

\bigskip 

\begin{example}
\label{degenere} Let 
\begin{equation}
h(\xi )=\sqrt{\sum_{i=1}^{n}|\xi _{i}|^{2p_{i}}},\qquad p_{i}\geq 1,
\label{h}
\end{equation}%
$p=\min_{i}p_{i}$ and $q=\max_{i}p_{i}\leq 2$, and 
\begin{equation}
\overline{Q}\left( \xi ,\lambda \right) =\sum_{i,j=1}^{n}h_{\xi _{i}\xi
_{j}}(\xi )\lambda _{i}\lambda _{j}\,,\;\;\;\;\forall \;\lambda ,\xi \in 
\mathbb{R}^{n}.  \label{Q segnato}
\end{equation}%
We prove that the associated quadratic form to $h$ is semidefinite, i.e. $%
\overline{Q}\left( \xi ,\lambda \right) \geq 0$. In fact 
\begin{equation}
\begin{split}
\overline{Q}\left( \xi ,\lambda \right) \left( \sum_{k=1}^{n}|\xi
_{k}|^{2p_{k}}\right) ^{\frac{3}{2}}=& -\sum_{i,j=1}^{n}\mathrm{sign}(\xi
_{i}\xi _{j})p_{i}p_{j}|\xi _{i}|^{2p_{i}-1}|\xi _{j}|^{2p_{j}-1}\lambda
_{i}\lambda _{j} \\
& +\sum_{k=1}^{n}|\xi _{k}|^{2p_{k}}\sum_{i=1}^{n}(2p_{i}^{2}-p_{i})|\xi
_{i}|^{2p_{i}-2}\lambda _{i}^{2}.
\end{split}
\label{Q segnato espressione}
\end{equation}%
By proceeding as above we have 
\begin{equation*}
\overline{Q}\left( \xi ,\lambda \right) \left( \sum_{k=1}^{n}|\xi
_{k}|^{2p_{k}}\right) ^{\frac{3}{2}}=-(v\cdot w)^{2}+\sum_{k=1}^{n}|\xi
_{k}|^{2p_{k}}\sum_{i=1}^{n}(2p_{i}^{2}-p_{i})|\xi _{i}|^{2p_{i}-2}\lambda
_{i}^{2},
\end{equation*}%
$v_{i}=p_{i}|\xi _{i}|^{p_{i}-1}\lambda _{i}$ and $w_{i}=\mathrm{sign}(\xi
_{i})|\xi _{i}|^{p_{i}}$. In this case the quadratic form $\overline{Q}$ is
degenerate ($\overline{Q}\left( \xi ,\lambda \right) =0$ if $(\xi \cdot
\lambda )=0$) but positive semidefinite: 
\begin{equation*}
\overline{Q}\left( \xi ,\lambda \right) \left( \sum_{k=1}^{n}|\xi
_{k}|^{2p_{k}}\right) ^{\frac{3}{2}}\geq \sum_{k=1}^{n}|\xi
_{k}|^{2p_{k}}\sum_{i=1}^{n}(p_{i}^{2}-p_{i})|\xi _{i}|^{2p_{i}-2}\lambda
_{i}^{2}\geq 0.
\end{equation*}%
On the other hand, if $\max_i \{|\xi _{i}|\}\geq 1$, 
\begin{equation*}
\begin{split}
\overline Q(\xi\,,\lambda)\left(\sum_{k=1}^{n}|\xi_k|^{2p_k}\right)^\frac 12
& \leq \sum_{i=1}^n (2p_i^2-p_i)|\xi_i|^{2p_i-2}\lambda_i^2 = \sum_{i=1}^n
(2p_i^2-p_i)(|\xi_i|^{2p_i})^{1-\frac 1{p_i}}\lambda_i^2 \\
& \leq (2q^2-q) \sum_{i=1}^n \left(\sum_{k=1}^n
|\xi_k|^{2p_k}\right)^{1-\frac 1{p_i}}\lambda_i^2
\end{split}%
\end{equation*}
since $\sum_{k=1}^n |\xi_k|^{2p_k}\geq \left(\max_i \{|\xi
_{i}|\}\right)^{2p}\geq 1$, 
\begin{equation*}
\overline Q(\xi\,,\lambda)\left(\sum_{k=1}^{n}|\xi_k|^{2p_k}\right)^\frac 12
\leq (2q^2-q)\left(\sum_{k=1}^n |\xi_k|^{2p_k}\right)^{1-\frac
1{q}}|\lambda|^2.
\end{equation*}
Now, again using $\max_i \{|\xi _{i}|\}\geq 1$, 
\begin{equation*}
\sum_{k=1}^n |\xi_k|^{2p_k}\geq (\max_i \{|\xi _{i}|\})^{2p}
\quad\Longrightarrow\quad \left(\sum_{k=1}^n |\xi_k|^{2p_k}\right)^{\frac
12-\frac 1{q}}\leq \left(\max_i \{|\xi _{i}|\}^{2p}\right)^{\frac 12-\frac
1{q}},
\end{equation*}
therefore 
\begin{equation*}
\begin{split}
\overline Q(\xi\,,\lambda) & \leq (2q^2-q)\left(\sum_{k=1}^n
|\xi_k|^{2p_k}\right)^{\frac 12-\frac 1{q}}|\lambda|^2 \leq
(2q^2-q)\left(\max_i \{|\xi _{i}|\}\right)^{p\frac {q-2}q}|\lambda|^2 \\
& \leq C|\xi|^{p\frac {q-2}q}|\lambda|^2.
\end{split}%
\end{equation*}
In this case $\overline{Q}\left( \xi ,\lambda \right) \leq C|\xi
|^{q-2}|\lambda |^{2}$ when $q\geq 2$. We denote by 
\begin{equation}
s=\frac{p}{q}(q-2)+2  \label{s}
\end{equation}%
with $1<p=\min_{i}p_{i}\leq q=\max_{i}p_{i}\leq 2$. We consider the
ellipticity conditions \eqref{ellipticity conditions under p,q}, with $s$
replaced by $q$, for the function 
\begin{equation}
f(\xi )=|\xi |^{p}+h(\xi ).  \label{s-p}
\end{equation}%
Since 
\begin{equation*}
\frac{s}{p}<1+\frac{2}{n}\quad \iff \quad \frac{q}{p}<1+\frac{q}{n}\,,
\end{equation*}%
from Corollary~\ref{theorem on p,q} we obtain the proof of a further
regularity result for the following energy integral 
\begin{equation}
F_{3}(u)=\int_{\Omega }{|Du|^{p}+\left( \sum_{i=1}^{n}\left\vert
u_{x_{i}}\right\vert ^{2p_{i}}\right) ^{\frac{1}{2}}}dx.  \label{F3}
\end{equation}

\begin{corollary}
\label{ex2-sec4} If $1<p=\min_{i}p_{i}\leq q=\max_{i}p_{i}\leq 2$ satisfy 
\begin{equation}
\frac{q}{p}<1+\frac{q}{n}\quad \iff \quad q<p^{\ast }=:\frac{np}{n-p}\,,
\label{condition ex2}
\end{equation}%
then any local minimizers of $F_{3}$ in \eqref{F3} is locally Lipschitz
continuous in $\Omega $.
\end{corollary}

We can consider also different integrands related with $h$ in \eqref{h}. By
taking in account Example~\ref{log} and Example~\ref{loglog}, we can
consider, for $1<q\leq 2$, $s=3-\frac{2}{q}\geq q$, 
\begin{equation}
f(\xi )=|\xi |(\log |\xi |)^{a}+\sqrt{\sum_{i=1}^{n-1}|\xi _{i}|^{2}+|\xi
_{n}|^{2q}},\qquad a>0,  \label{h+log}
\end{equation}%
or 
\begin{equation}
f(\xi )=|\xi |L_{k}\left( |\xi |\right) +\sqrt{\sum_{i=1}^{n-1}|\xi
_{i}|^{2}+|\xi _{n}|^{2q}},  \label{h+loglog}
\end{equation}%
Assumption \eqref{H} holds with $g_{2}(t)=Ct^{s-2}$ and respectively 
\begin{equation*}
g_{1}(t)=c\,\frac{(\log t)^{a-1}}{t}\quad \text{ or }\quad g_{1}(t)=\frac{c}{%
(1+t)(1+L_{1}(t))\cdots (1+L_{k-1}(t))}\,,
\end{equation*}%
\begin{equation*}
\mu =2-s,\quad \beta >\frac{n-2}{2n}s-\frac{1}{2}+\frac{2}{n},\quad \alpha
>s\,.
\end{equation*}%
Therefore, if $s<1+\frac{2}{n}$, by Theorem~\ref{superlinear} the
corresponding local minimizers are locally Lipschitz continuous.
\end{example}

\section{Interpolation lemma}

\label{Preliminary lemma}

As usual we denote by $B_{R}$ a generic ball of radius $R$ compactly
contained in $\Omega $ and by $B_{\varrho }$ a ball of radius $\varrho <R$
concentric with $B_{R}$.

\begin{lemma}[interpolation]
\label{interpolation} Let $v\in L_{\mathrm{loc}}^{\infty }\left( \Omega
\right) $ and let us assume that for some $\vartheta \geq 1$, $c>0$ and for
every $\varrho $\ and $R$\ such that $0<\rho <R$%
\begin{equation}
\left\Vert v\right\Vert _{L^{\infty }\left( B_{\varrho }\right) }^{\frac{1}{%
\vartheta }}\leq \frac{c}{\left( R-\varrho \right) ^{n}}\int_{B_{R}}\left%
\vert v\right\vert \,dx\,.  \label{assumption in the Lemma}
\end{equation}%
Then, for every $\lambda \in \left( \frac{\vartheta -1}{\vartheta },1\right) 
$ (i.e., in particular with $\vartheta \left( 1-\lambda \right) <1$) there
exists a constant $c_{\lambda }$ such that, for every $\varrho <R$, 
\begin{equation}
\left\Vert v\right\Vert _{L^{\infty }\left( B_{\varrho }\right) }^{\frac{%
1-\vartheta \left( 1-\lambda \right) }{\vartheta }}\leq \frac{c_{\lambda }}{%
\left( R-\varrho \right) ^{n}}\int_{B_{R}}\left\vert v\right\vert ^{\lambda
}\,dx.  \label{interpolation bound 2}
\end{equation}
\end{lemma}

\begin{proof}
Fixed $\lambda \in \left( \frac{\vartheta -1}{\vartheta },1\right) $, we
make use of the interpolation inequality 
\begin{equation*}
\int_{B_{\varrho }}\left\vert v\right\vert \,dx=\int_{B_{\varrho
}}\left\vert v\right\vert ^{1-\lambda }\left\vert v\right\vert ^{\lambda
}\,dx\leq \left\Vert v\right\Vert _{L^{\infty }\left( B_{\varrho }\right)
}^{1-\lambda }\int_{B_{\varrho }}\left\vert v\right\vert ^{\lambda }\,dx\,.
\end{equation*}%
By the assumption (\ref{assumption in the Lemma}) we obtain 
\begin{equation*}
\int_{B_{\varrho }}\left\vert v\right\vert \,dx \leq \left( \frac{c}{\left(
R-\varrho \right) ^{n}}\int_{B_{R}}\left\vert v\right\vert \,dx\right)
^{\vartheta \left( 1-\lambda \right) }\int_{B_{\varrho }}\left\vert
v\right\vert ^{\lambda }\,dx\,.
\end{equation*}%
We denote by $\gamma :=\vartheta \left( 1-\lambda \right) $ and we observe
that $0<\gamma <1$ since $\lambda >\frac{\vartheta -1}{\vartheta }$. Thus
the previous estimate has the equivalent form 
\begin{equation}
\int_{B_{\varrho }}\left\vert v\right\vert \,dx\leq c^{\gamma
}\int_{B_{\varrho }}\left\vert v\right\vert ^{\lambda }\,dx\cdot \left( 
\frac{1}{\left( R-\varrho \right) ^{n}}\int_{B_{R}}\left\vert v\right\vert
\,dx\right) ^{\gamma }.  \label{interpolation 1}
\end{equation}%
Given $\varrho _{0}$ and $R_{0}$, with $0<\varrho _{0}<R_{0}\leq \varrho
_{0}+1$, we define a decreasing sequence $\varrho _{k}$ by $\varrho
_{k}=R_{0}-\frac{R_{0}-\varrho _{0}}{2^{k}}\,$, $k=0,1,2,\ldots $ In (\ref%
{interpolation 1}) we pose $\varrho =\varrho _{k}$ and $R=\varrho _{k+1}$.
Since $R-\varrho =\varrho _{k+1}-\varrho _{k}=\frac{R_{0}-\varrho _{0}}{%
2^{k+1}}$, we obtain 
\begin{equation*}
\int_{B_{\varrho _{k}}}\left\vert v\right\vert \,dx\leq c^{\gamma
}\int_{B_{R_{0}}}\left\vert v\right\vert ^{\lambda }\,dx\cdot \left( \frac{%
2^{n\left( k+1\right) }}{\left( R_{0}-\varrho _{0}\right) ^{n}}%
\int_{B_{\varrho _{k+1}}}\left\vert v\right\vert \,dx\right) ^{\gamma },
\end{equation*}%
Denote $B_{k}=\int_{B_{\varrho _{k}}}\left\vert v\right\vert \,dx$ for $%
k=0,1,2,\ldots $. The last inequality becomes 
\begin{equation*}
B_{k}\leq c^{\gamma }\int_{B_{R_{0}}}\left\vert v\right\vert ^{\lambda
}\,dx\cdot \frac{2^{n\gamma \left( k+1\right) }}{\left( R_{0}-\varrho
_{0}\right) ^{n\gamma }}B_{k+1}^{\gamma }\;.
\end{equation*}%
We start to iterate with $k=0,1,2,\ldots $%
\begin{equation*}
B_{0}\leq c^{\gamma }\int_{B_{R_{0}}}\left\vert v\right\vert ^{\lambda
}\,dx\cdot \frac{2^{n\gamma }}{\left( R_{0}-\varrho _{0}\right) ^{n\gamma }}%
B_{1}^{\gamma }\;
\end{equation*}%
\begin{equation*}
\leq c^{\gamma }\int_{B_{R_{0}}}\left\vert v\right\vert ^{\lambda }\,dx\cdot 
\frac{2^{n\gamma }}{\left( R_{0}-\varrho _{0}\right) ^{n\gamma }}\left(
c^{\gamma }\int_{B_{R_{0}}}\left\vert v\right\vert ^{\lambda }\,dx\cdot 
\frac{2^{n\gamma \cdot 2}}{\left( R_{0}-\varrho _{0}\right) ^{n\gamma }}%
B_{2}^{\gamma }\right) ^{\gamma }\;
\end{equation*}%
and for general $k=1,2,3,\ldots $ we have 
\begin{equation*}
B_{0}\leq \left( \frac{c^{\gamma }\int_{B_{R_{0}}}\left\vert v\right\vert
^{\lambda }\,dx}{\left( R_{0}-\varrho _{0}\right) ^{n\gamma }}\right)
^{\sum_{i=0}^{k-1}\gamma ^{i}}\left( 2^{n}\right) ^{\sum_{i=1}^{k}i\gamma
^{i}}\left( B_{k}\right) ^{\gamma ^{k}}\;.
\end{equation*}%
Since $0<\gamma <1$, passing to the limit as $k\to\infty$, $%
\sum_{i=0}^{\infty}i\gamma ^{i}<\infty$ and $\sum_{i=0}^{\infty}\gamma ^{i}=%
\frac{1}{1-\gamma }$. Moreover the increasing sequence $B_{k}=\int_{B_{%
\varrho _{k}}}\left\vert v\right\vert \,dx$ is bounded by $%
\int_{B_{R_{0}}}\left\vert v\right\vert \,dx$ for $k=0,1,2,\ldots $ Thus $%
\left( B_{k}\right) ^{\gamma ^{k}}=\left( \int_{B_{\varrho _{k}}}\left\vert
v\right\vert \,dx\right) ^{\gamma ^{k}}\leq \left(
\int_{B_{R_{0}}}\left\vert v\right\vert \,dx\right) ^{\gamma ^{k}}$ and the
right hand side converges to $1$ as $k\rightarrow \infty $. Therefore, in
the limit as $k\rightarrow \infty $, there exists a constant $c_{1}$ such
that 
\begin{equation}
B_{0}=\int_{B_{\varrho _{0}}}\left\vert v\right\vert \,dx\leq c_{1}\left( 
\frac{1}{\left( R_{0}-\varrho _{0}\right) ^{n\gamma }}\int_{B_{R_{0}}}\left%
\vert v\right\vert ^{\lambda }\,dx\right) ^{\frac{1}{1-\gamma }}\,.
\label{interpolation 11}
\end{equation}%
Fixed $\varrho <R$ we consider $\overline{\varrho }=\frac{R+\varrho }{2}$
and, by combining the assumption (\ref{assumption in the Lemma}) and (\ref%
{interpolation 11}), since $R-\overline{\varrho }=\overline{\varrho }%
-\varrho $ and $\gamma =\vartheta \left( 1-\lambda \right) $, 
\begin{equation*}
\left\Vert v\right\Vert _{L^{\infty }\left( B_{\varrho }\right) }\leq \left( 
\frac{c}{\left( \overline{\varrho }-\varrho \right) ^{n}}\int_{B_{\overline{%
\varrho }}}\left\vert v\right\vert \,dx\right) ^{\vartheta }
\end{equation*}%
\begin{equation*}
\leq \left( \frac{c\cdot c_{1}}{\left( \overline{\varrho }-\varrho \right)
^{n}}\left( \frac{1}{\left( R-\overline{\varrho }\right) ^{n\gamma }}%
\int_{B_{R}}\left\vert v\right\vert ^{\lambda }\,dx\right) ^{\frac{1}{%
1-\gamma }}\right) ^{\vartheta }
\end{equation*}%
\begin{equation*}
\leq c_{2}\left( \frac{1}{\left( R-\overline{\varrho }\right) ^{n\left(
1-\gamma \right) +n\gamma }}\int_{B_{R}}\left\vert v\right\vert ^{\lambda
}\,dx\right) ^{\frac{\vartheta }{1-\gamma }}=c_{3}\left( \frac{1}{\left(
R-\varrho \right) ^{n}}\int_{B_{R}}\left\vert v\right\vert ^{\lambda
}\,dx\right) ^{\frac{\vartheta }{1-\vartheta \left( 1-\lambda \right) }},
\end{equation*}%
which gives (\ref{interpolation bound 2}) and the proof of the Lemma is
concluded.
\end{proof}

\medskip 

\bigskip

\section{A priori estimates}

\label{A priori estimates}

In order to simplify the notations, without loss of generality in this
section we assume that $t_{0}=1$. First of all we give a technical result.

\begin{lemma}
\label{g1eg2} Let us assume that \eqref{H}$_2$ and \eqref{H}$_3$ hold. Then
for every $\gamma\geq 0$ there exists a constant $C_3= C_3(C_1\,,g_2(1))>0$
independent of $\gamma$, such that 
\begin{equation}  \label{eq:int1}
C_3\left[ 1+ g_2(1 + t)^{\frac{1}{2^*}} \frac{(1 + t)^{\frac{\gamma}{2}%
+1-\beta}}{ \left (\frac{\gamma}{2} + 1 - \beta \right )^2} \right] \leq
1+\int_0^t (1+s)^{\frac{\gamma - 2}{2}} s \sqrt{g_1(1+s)}\,ds
\end{equation}
for every $t\geq 0$, where , for $n>2$, $2^{\ast }=\frac{2n}{n-2}$, while,
for $n=2$, $2^{\ast }$ can be any number greater than $\frac 2{1-\beta}$.
\end{lemma}

\begin{proof}
If $t\geq 0$ then, by assumption \eqref{H}$_{3}$ 
\begin{equation*}
1+\int_{0}^{t}(1+s)^{\frac{\gamma -2}{2}}s\sqrt{g_{1}(1+s)}\,ds\geq
\,1+\int_{0}^{t}(1+s)^{\frac{\gamma -2}{2}}s\,\frac{1}{\sqrt{C_{1}}}%
(1+s)^{-\beta }g_{2}(1+s)^{\frac{1}{2^{\ast }}}\,ds.
\end{equation*}%
On the other hand, since $g_{2}$ is decreasing, $g_{2}(1+s)^{\frac{1}{%
2^{\ast }}}\geq \,g_{2}(1+t)^{\frac{1}{2^{\ast }}}$ and therefore 
\begin{equation*}
1+\frac{1}{\sqrt{C_{1}}}g_{2}(1+t)^{\frac{1}{2^{\ast }}}\int_{0}^{t}(1+s)^{%
\frac{\gamma -2}{2}-\beta }s\,ds\leq 1+\int_{0}^{t}(1+s)^{\frac{\gamma -2}{2}%
}s\,\frac{1}{\sqrt{C_{1}}}(1+s)^{-\beta }g_{2}(1+s)^{\frac{1}{2^{\ast }}%
}\,ds.
\end{equation*}%
By Lemma 2.2 in \cite{EMM3}, we have that (see (2.6) here): let $\alpha
_{0}>0$ there exists a constant $c$ depending on $\alpha _{0}$, but
independent of $\alpha \geq \alpha _{0}$, such that 
\begin{equation}
(1+t)^{\alpha }\leq c\,\alpha ^{2}\left( 1+\int_{0}^{t}(1+s)^{\alpha
-2}s\,ds\right) .  \label{lemmapaolo}
\end{equation}%
In our case $\alpha :=\frac{\gamma -2}{2}-\beta +2=\frac{\gamma }{2}+1-\beta 
$ and $\alpha \geq \alpha _{0}:=\frac{2}{2^{\ast }}$. Inequality %
\eqref{lemmapaolo} is valid for all $t\geq 0$ so in particular for $t\geq 1$
and it entails 
\begin{equation*}
\int_{0}^{t}(1+s)^{\alpha -2}s\,ds\geq \,\frac{(1+t)^{\alpha }}{c\,\alpha
^{2}}-1.
\end{equation*}%
The last inequality implies 
\begin{eqnarray*}
&&1+\frac{1}{\sqrt{C_{1}}}g_{2}(1+t)^{\frac{1}{2^{\ast }}}\int_{0}^{t}(1+s)^{%
\frac{\gamma -2}{2}-\beta }s\,ds \\
&\geq &\,1+\frac{1}{\sqrt{C_{1}}}g_{2}(1+t)^{\frac{1}{2^{\ast }}}\left[ 
\frac{(1+t)^{\frac{\gamma }{2}+1-\beta }}{c\left( \frac{\gamma }{2}+1-\beta
\right) ^{2}}-1\right] \\
&=&1+\frac{1}{\sqrt{C_{1}}}g_{2}(1+t)^{\frac{1}{2^{\ast }}}\frac{(1+t)^{%
\frac{\gamma }{2}+1-\beta }}{c\left( \frac{\gamma }{2}+1-\beta \right) ^{2}}-%
\frac{1}{\sqrt{C_{1}}}g_{2}(1+t)^{\frac{1}{2^{\ast }}}.
\end{eqnarray*}%
Now we observe that, for every $t\geq 0$, since $g_{2}$ is decreasing, 
\begin{equation*}
\frac{1}{\sqrt{C_{1}}}g_{2}(1+t)^{\frac{1}{2^{\ast }}}\leq \,\frac{1}{\sqrt{%
C_{1}}}g_{2}(1)^{\frac{1}{2^{\ast }}}=:\tilde{C}_{1}.
\end{equation*}%
Thus summing up 
\begin{equation*}
1+\tilde{C}_{1}+\int_{0}^{t}(1+s)^{\frac{\gamma -2}{2}}s\sqrt{g_{1}(1+s)}%
\,ds\geq \,1+\frac{1}{\sqrt{C_{1}}}g_{2}(1+t)^{\frac{1}{2^{\ast }}}\frac{%
(1+t)^{\frac{\gamma }{2}+1-\beta }}{c\left( \frac{\gamma }{2}+1-\beta
\right) ^{2}}
\end{equation*}%
which in turn implies 
\begin{eqnarray*}
&&(1+\tilde{C}_{1})\left[ 1+\int_{0}^{t}(1+s)^{\frac{\gamma -2}{2}}s\sqrt{%
g_{1}(1+s)}\,ds\right] \\
&\geq &\,1+\tilde{C}_{1}+\int_{0}^{t}(1+s)^{\frac{\gamma -2}{2}}s\sqrt{%
g_{1}(1+s)}\,ds \\
&\geq &\,1+\frac{1}{\sqrt{C_{1}}}g_{2}(1+t)^{\frac{1}{2^{\ast }}}\frac{%
(1+t)^{\frac{\gamma }{2}+1-\beta }}{c\left( \frac{\gamma }{2}+1-\beta
\right) ^{2}}
\end{eqnarray*}%
therefore, by setting $\tilde{C}_{2}:=(1+\tilde{C}_{1})\sqrt{C_{1}}\,c$, we
get \eqref{eq:int1} for $C_{3}=\frac{1}{\tilde{C}_{2}}$. We note explicitly
that $\tilde{C}_{2}$ may depend on $n$ but it is independent of $\gamma $.
\end{proof}

\begin{lemma}
\label{step 1} Assume that $f$ satisfies the growth assumptions \eqref{H}$%
_{1}$, \eqref{H}$_{2}$, \eqref{H}$_{3}$. In addiction, assume that $f\left(
\xi \right) $ is of class $\mathcal{C}^{2}(\mathbb{R}^{n})$ and for every $%
M>0$ there exists a positive constant $\ell =\ell (M)$ such that 
\begin{equation}
\ell \,|\lambda |^{2}\leq \,\sum_{i,j=1}^{n}f_{\xi _{i}\xi _{j}}(\xi
)\,\lambda _{i}\,\lambda _{j}\qquad \forall \lambda ,\xi \in \mathbb{R}%
^{n},\,|\xi |\geq M.  \label{supp1}
\end{equation}%
If $u\in W_{\mathrm{loc}}^{1,\infty }(\Omega )$ is a a local minimizer of %
\eqref{energy integral}, then for every $0<\rho <R$, $\bar{B}_{R}\subset
\Omega $ there exists a positive constants $c_{4}$ depending only on $C_{1}$%
, $\beta $, $g_{2}(1)$, such that 
\begin{equation}
\begin{split}
& \left( \Vert 1+(|Du|-1)_{+}\Vert _{L^{\infty }(B_{\rho })}\right)
^{2-n\beta } \\
& \qquad \qquad \qquad \leq \,\frac{c_{4}}{({R}-{\rho })^{n}}%
\int_{B_{R}}(1+(|Du|-1)_{+})^{2}g_{2}(1+(|Du|-1)_{+})\,dx.
\end{split}
\label{eq:normainfinto}
\end{equation}
\end{lemma}

\begin{proof}
Since the local minimizer $u$ is in $W_{\mathrm{loc}}^{1,\infty }(\Omega )$,
it satisfies the Euler equation: for every open set $\Omega ^{\prime }$
compactly contained in $\Omega $ we have 
\begin{equation*}
\int_{\Omega }\sum_{i=1}^{n}f_{\xi _{i}}(Du)\,\varphi _{x_{i}}\,dx=0\qquad
\forall \varphi \in W_{0}^{1,2}(\Omega ^{\prime }).
\end{equation*}%
Moreover, by the techniques of the difference quotient (see for example \cite%
[Ch.~8, Sect.~8.1]{giusti}), $u\in W_{\mathrm{loc}}^{2,2}(\Omega )$, then
the second variation holds: 
\begin{equation*}
\int_{\Omega }\sum_{i,j=1}^{n}f_{\xi _{i}\xi _{j}}(Du)u_{x_{j}x_{k}}\varphi
_{x_{i}}\,dx=0,\qquad \forall k=1,\dots ,n,\,\,\,\forall \varphi \in
W_{0}^{1,2}(\Omega ^{\prime }).
\end{equation*}%
For fixed $k=1,\dots ,n$ let $\eta \in \mathcal{C}_{0}^{1}(\Omega ^{\prime
}) $ be equal to 1 in $B_{\rho }$, with support contained in $B_{R}$, such
that $|D\eta |\leq \,\frac{2}{(R-\rho )},$ and consider $\varphi =\eta
^{2}\,u_{x_{k}}\,\Phi ((|Du|-1)_{+})$ with $\Phi $ non negative, increasing,
locally Lipschitz continuous on $[0,+\infty )$, such that $\Phi (0)=0$. Here 
$(a)_{+}$ denotes the positive part of $a\in \mathbb{R}$; in the following
we denote $\Phi ((|Du|-1)_{+})=\Phi (|Du|-1)_{+}.$ Then a.e. in $\Omega $ 
\begin{equation*}
\varphi _{x_{i}}=2\eta \,\eta _{x_{i}}u_{x_{k}}\Phi (|Du|-1)_{+}+\eta
^{2}u_{x_{i}x_{k}}\Phi (|Du|-1)_{+}+\eta ^{2}u_{x_{k}}\Phi ^{\prime
}(|Du|-1)_{+}[(|Du|-1)_{+}]_{x_{i}}.
\end{equation*}%
Proceeding along the lines of \cite{mar96}, we therefore deduce that 
\begin{eqnarray*}
&&\int_{\Omega }2\eta \Phi (|Du|-1)_{+}\sum_{i,j=1}^{n}f_{\xi _{i}\xi
_{j}}(Du)u_{x_{j}x_{k}}\eta _{x_{i}}u_{x_{k}}\,dx \\
&&+\int_{\Omega }\eta ^{2}\Phi (|Du|-1)_{+}\sum_{i,j=1}^{n}f_{\xi _{i}\xi
_{j}}(Du)u_{x_{j}x_{k}}u_{x_{i}x_{k}}\,dx \\
&&+\int_{\Omega }\eta ^{2}\Phi ^{\prime }(|Du|-1)_{+}\sum_{i,j=1}^{n}f_{\xi
_{i}\xi _{j}}(Du)u_{x_{j}x_{k}}u_{x_{k}}[(|Du-1|)_{+}]_{x_{i}}\,dx=0.
\end{eqnarray*}%
We estimate the first integral in the previous equation by using the
Cauchy-Schwarz inequality and the Young inequality so that 
\begin{equation*}
\begin{split}
& \left\vert \int_{\Omega }2\eta \Phi (|Du|-1)_{+}\sum_{i,j=1}^{n}f_{\xi
_{i}\xi _{j}}(Du)u_{x_{j}x_{k}}\eta _{x_{i}}u_{x_{k}}\,dx\right\vert \\
& \quad \leq \int_{\Omega }2\Phi (|Du|-1)_{+}\left( \eta
^{2}\sum_{i,j=1}^{n}f_{\xi _{i}\xi
_{j}}(Du)u_{x_{i}x_{k}}u_{x_{j}x_{k}}\right) ^{\frac{1}{2}}\left(
\sum_{i,j=1}^{n}f_{\xi _{i}\xi _{j}}(Du)\eta _{x_{i}}u_{x_{k}}\eta
_{x_{j}}u_{x_{k}}\right) ^{\frac{1}{2}}\,dx \\
& \quad \leq \frac{1}{2}\int_{\Omega }\eta ^{2}\Phi
((|Du|-1)_{+})\sum_{i,j=1}^{n}f_{\xi _{i}\xi
_{j}}(Du)u_{x_{i}x_{k}}u_{x_{j}x_{k}}\,dx \\
& \qquad +2\int_{\Omega }\Phi ((|Du|-1)_{+})\sum_{i,j=1}^{n}f_{\xi _{i}\xi
_{j}}(Du)\eta _{x_{i}}u_{x_{k}}\eta _{x_{j}}u_{x_{k}}\,dx.
\end{split}%
\end{equation*}%
Therefore we deduce 
\begin{equation*}
\begin{split}
\frac{1}{2}\int_{\Omega }\eta ^{2}\Phi (|Du|-1)_{+}& \sum_{i,j=1}^{n}f_{\xi
_{i}\xi _{j}}(Du)u_{x_{i}x_{k}}u_{x_{j}x_{k}}\,dx \\
& +\int_{\Omega }\eta ^{2}\Phi ^{\prime }(|Du|-1)_{+}\sum_{i,j=1}^{n}f_{\xi
_{i}\xi _{j}}(Du)u_{x_{k}}[(|Du-1|)_{+}]_{x_{i}}\,dx \\
& \leq 2\int_{\Omega }\Phi (|Du|-1)_{+}\sum_{i,j=1}^{n}f_{\xi _{i}\xi
_{j}}(Du)\eta _{x_{i}}u_{x_{k}}\eta _{x_{j}}u_{x_{k}}\,dx.
\end{split}%
\end{equation*}%
Since a.e.~in $\Omega $ 
\begin{equation*}
\lbrack (|Du|-1)_{+}]_{x_{i}}=%
\begin{cases}
(|Du|)_{x_{i}}=\frac{1}{|Du|}\sum_{k}u_{x_{i}x_{k}}u_{x_{k}} & \text{ if $%
|Du|>1$,} \\ 
0 & \text{ if $|Du|\leq 1$,}%
\end{cases}%
\end{equation*}%
by summing up in the previous chain of inequalities with respect to $%
k=1,\dots ,n$ we obtain 
\begin{equation*}
\sum_{k=1}^{n}\sum_{i,j=1}^{n}f_{\xi _{i}\xi
_{j}}(Du)u_{x_{j}x_{k}}u_{x_{k}}[(|Du|-1)_{+}]_{x_{i}}=|Du|%
\sum_{i,j=1}^{n}f_{\xi _{i}\xi
_{j}}(Du)[(|Du-1|)_{+}]_{x_{j}}[(|Du|-1)_{+}]_{x_{i}}
\end{equation*}%
therefore we deduce the estimate 
\begin{equation*}
\begin{split}
\int_{\Omega }\eta ^{2}& \Phi (|Du|-1)_{+}\sum_{k,i,j=1}^{n}f_{\xi _{i}\xi
_{j}}(Du)u_{x_{j}x_{k}}u_{x_{i}x_{k}}\,dx \\
& +\int_{\Omega }\eta ^{2}|Du|\Phi ^{\prime
}(|Du|-1)_{+}\sum_{i,j=1}^{n}f_{\xi _{i}\xi
_{j}}(Du)[(|Du-1|)_{+}]_{x_{j}}[(|Du-1|)_{+}]_{x_{i}}\,dx \\
& \leq 4\int_{\Omega }\Phi (|Du|-1)_{+}\sum_{k,i,j=1}^{n}f_{\xi _{i}\xi
_{j}}(Du)\eta _{x_{i}}u_{x_{k}}\eta _{x_{j}}u_{x_{k}}\,dx.
\end{split}%
\end{equation*}%
Using the inequality $|D(|Du|-1)_{+}|^{2}\leq \,|D^{2}u|^{2}$ and the
ellipticity condition in \eqref{H}$_{1}$ we obtain 
\begin{equation}
\begin{split}
\int_{\Omega }& \eta ^{2}[\Phi (|Du|-1)_{+}+|Du|\Phi ^{\prime }(|Du|-1)_{+}]{%
g_{1}(1+(|Du|-1)_{+}))}\,|D(|Du|-1)_{+}|^{2}\,dx \\
& =\int_{\Omega }\eta ^{2}[\Phi (|Du|-1)_{+}+|Du|\Phi ^{\prime }(|Du|-1)_{+}]%
{g_{1}(|Du|)}\,|D(|Du|-1)_{+}|^{2}\,dx \\
& \leq \,4\,\int_{\Omega }|D\eta |^{2}\,\Phi
(|Du|-1)_{+}\,g_{2}(|Du|)\,|Du|^{2}\,dx \\
& =\,4\,\int_{\Omega }|D\eta |^{2}\,\Phi
(|Du|-1)_{+}\,g_{2}(1+(|Du|-1)_{+}))\,|Du|^{2}\,dx.
\end{split}
\label{(25)Jota}
\end{equation}%
Let us define 
\begin{equation}
G(t)=1+\int_{0}^{t}\sqrt{\Phi (s)\,g_{1}(1+s)}\,ds\qquad \forall t\geq \,0.
\label{defG}
\end{equation}%
By Jensen's inequality and the monotonicity of $\Phi $, since $t\mapsto
tg_{2}(t)$ is increasing, 
\begin{equation*}
\begin{split}
G(t)& =1+\int_{0}^{t}\sqrt{\Phi (s)(1+s)g_{1}(1+s)\frac{1}{1+s}}\,ds\leq
1+\int_{0}^{t}\sqrt{\Phi (s)(1+s)g_{2}(1+s)\frac{1}{1+s}}\,ds \\
& \leq 1+\sqrt{\Phi (t)(1+t)g_{2}(1+t)}\int_{0}^{t}\frac{1}{\sqrt{1+s}}%
\,ds\leq 1+2\sqrt{\Phi (t)(1+t)g_{2}(1+t)}\,\sqrt{1+t},
\end{split}%
\end{equation*}%
hence $[G(t)]^{2}\leq 8\left[ 1+\Phi (t)(1+t)^{2}g_{2}(1+t)\right] $. On the
other hand 
\begin{equation*}
\begin{split}
|D& (\eta (G((|Du|-1)_{+})|^{2} \\
& \leq \,2\,|D\eta |^{2}[G((|Du|-1)_{+})]^{2}+2\eta ^{2}[G^{\prime
}((|Du|-1)_{+})]^{2}\,|D((|Du|-1)_{+})|^{2} \\
& \leq 16\,|D\eta |^{2}\,(1+\Phi (|Du|-1)_{+}\,g_{2}(|Du|)|Du|^{2})+2\,\eta
^{2}\,\Phi (|Du|-1)_{+}g_{1}(|Du|)\,|D(|Du|)|^{2}.
\end{split}%
\end{equation*}%
Since $\Phi (|Du\left( x\right) |-1)_{+}=0$ when $|Du\left( x\right) |\leq 1$%
, by \eqref{(25)Jota} we get 
\begin{equation}
\begin{split}
& \int_{\Omega }|D(\eta \,G((|Du|-1)_{+})|^{2}\,dx \\
& \leq \,24\int_{\Omega }|D\eta |^{2}\,(1+\Phi
(|Du|-1)_{+}g_{2}(|Du|)|Du|^{2})\,dx \\
& =24\int_{\Omega }|D\eta |^{2}\,(1+\Phi
(|Du|-1)_{+}\,g_{2}(1+(|Du|-1)_{+})\,(1+(|Du|-1)_{+})^{2}\,dx.
\end{split}
\label{Sobolev3}
\end{equation}%
Let us assume 
\begin{equation}
\Phi (t)=(1+t)^{\gamma -2}t^{2}\qquad \gamma \geq 0.  \label{defPhi}
\end{equation}%
By the Sobolev inequality, there exists a constant $c_{S}$ such that 
\begin{equation}
\left\{ \int_{\Omega }[\eta \,G((|Du|-1)_{+})]^{2^{\ast }}\,dx\right\}
^{2/2^{\ast }}\leq \,c_{S}\,\int_{\Omega }|D(\eta (G(|Du|-1)_{+}))|^{2}\,dx
\label{Sobolev1}
\end{equation}%
where $2^{\ast }=\frac{2n}{n-2}$ if $n>2$ and a number greater than $\frac{2%
}{1-\beta }$ if $n=2$. We apply \eqref{eq:int1} with the choice $%
t=(|Du|-1)_{+}$ 
\begin{eqnarray*}
G((|Du|-1)_{+}) &=&1+\int_{0}^{(|Du|-1)_{+}}(1+s)^{\frac{\gamma -2}{2}}s%
\sqrt{g_{1}(1+s)}\,ds \\
&\geq &C_{3}\left[ 1+g_{2}(1+(|Du|-1)_{+})^{\frac{1}{2^{\ast }}}\frac{%
(1+(|Du|-1)_{+})^{\frac{\gamma }{2}+1-\beta }}{\left( \frac{\gamma }{2}%
+1-\beta \right) ^{2}}\right]
\end{eqnarray*}%
thus by \eqref{Sobolev3} we obtain that there exists $c=c(C_{3})>0$ such
that, for all $\gamma \geq 0$, 
\begin{equation}
\begin{split}
& \left\{ \int_{\Omega }\eta ^{2^{\ast }}(1+(1+(|Du|-1)_{+})^{(\gamma
+2-2\beta )\frac{2^{\ast }}{2}}\,g_{2}(1+(|Du|-1)_{+}))\,dx\right\} ^{\frac{2%
}{2^{\ast }}} \\
& \qquad \leq \,16\,c\,\left( \frac{\gamma }{2}+1-\beta \right)
^{4}\,\int_{\Omega }|D\eta |^{2}\,(1+(1+(|Du|-1)_{+})^{\gamma
+2}\,g_{2}(1+(|Du|-1)_{+})\,dx \\
& \qquad \leq \,c\,\left( \gamma +2\right) ^{4}\,\int_{\Omega }|D\eta
|^{2}\, \left[ 1+(1+(|Du|-1)_{+})^{\gamma +2}\,g_{2}(1+(|Du|-1)_{+}\right]
\,dx
\end{split}%
\end{equation}%
where we used once more \eqref{Sobolev3} and \eqref{Sobolev1}. The iteration
process follows now the arguments contained in \cite{MP}; for the sake of
clarity we focus on the main steps. From now on, we label the constants;
this will be useful in the sequel. We set $\delta :=\left( \gamma +2\right) $
and we notice that, since $\gamma \geq 0$, then $\delta \geq 2$. Then 
\begin{equation}
\begin{split}
& \left\{ \int_{B_{\rho }}\left[ 1+(1+(|Du|-1)_{+})^{(\delta -2\beta )\frac{%
2^{\ast }}{2}}\,\,g_{2}(1+(|Du|-1)_{+}\right] \,dx\right\} ^{\frac{2}{%
2^{\ast }}} \\
& \qquad \leq \,c_{1}\,\left( \frac{\delta ^{2}}{R-\rho }\right)
^{2}\int_{B_{R}}\left[ 1+(1+(|Du|-1))_{+}^{\delta }g_{2}(1+(|Du|-1)_{+})%
\right] \,dx,
\end{split}
\label{MP1}
\end{equation}%
$c_{1}=c_{1}(C_{3})>0$, for all $\delta \geq 2$. We fix $\bar{\rho}$ and $%
\bar{R}$ such that $\bar{\rho}<\bar{R}$ and we introduce the decreasing
sequence of radii $\{\rho _{i}\}_{i\geq 0}$ 
\begin{equation*}
\rho _{i}=\bar{\rho}+\frac{\bar{R}-\bar{\rho}}{2^{i}},\qquad \forall i\geq 0,
\end{equation*}%
observing that $\bar{\rho}<\rho _{i+1}<\rho _{i}<\bar{R}=\rho _{0}$.
Correspondingly we define as well the increasing sequence of exponents $%
\left\{ \delta _{i}\right\} _{i\geq 0}$ such that 
\begin{equation*}
\delta _{0}=2\qquad \text{and}\qquad \delta _{i+1}=(\delta _{i}-2\beta )\,%
\frac{2^{\ast }}{2},\qquad \forall i\geq 0.
\end{equation*}%
First of all we check that $\delta _{i}\geq 2$ for all $i\geq 0$. By
induction this is equivalent to require $\beta <1-\frac{2}{2^{\ast }}$ that
is $2^{\ast }>\frac{2}{1-\beta }$ and this is always satisfied.

We can rewrite \eqref{MP1} with $\rho =\rho _{i+1}$, $R=\rho _{i}$, $\delta
=\delta _{i}$. For every $i\geq 0$ we then obtain 
\begin{equation*}
\begin{split}
& \left\{ \int_{B_{\rho _{i+1}}}\left[ 1+(1+(|Du|-1)_{+})^{\delta
_{i+1}}g_{2}((1+|Du|-1)_{+})\right] \,dx\right\} ^{\frac{2}{2^{\ast }}} \\
& \qquad \leq \,c_{1}\,\left( \frac{\delta _{i}^{2}\,2^{i+1}}{\bar{R}-\bar{%
\rho}}\right) ^{2}\int_{B_{\rho _{i}}}\left[ 1+(1+(|Du|-1)_{+})^{\delta
_{i}}g_{2}(1+(|Du|-1)_{+})\right] \,dx.
\end{split}%
\end{equation*}%
By iterating the previous inequality, we are able to deduce 
\begin{equation*}
\begin{split}
& \left\{ \int_{B_{\rho _{i+1}}}\left[ 1+(1+(|Du|-1)_{+})^{\delta
_{i+1}}g_{2}(1+(|Du|-1)_{+})\right] \,dx\right\} ^{\left( \frac{2}{2^{\ast }}%
\right) ^{i+1}} \\
& \qquad \leq \,c_{2}\,\int_{B_{\bar{R}}}\left[
1+(1+(|Du|-1)_{+})^{2}g_{2}(1+(|Du|-1)_{+})\right] \,dx,
\end{split}%
\end{equation*}%
where, by induction we computed 
\begin{equation}
\begin{split}
\delta _{i+1}& =2\left( \frac{2^{\ast }}{2}\right) ^{i+1}-2\beta
\sum_{k=1}^{i+1}\left( \frac{2^{\ast }}{2}\right) ^{k}=2\left( \frac{2^{\ast
}}{2}\right) ^{i+1}\left[ 1-\beta \sum_{k=0}^{i}\left( \frac{2}{2^{\ast }}%
\right) ^{k}\right] \\
& =2\left( \frac{2^{\ast }}{2}\right) ^{i+1}\left[ 1-\beta \frac{1-\left( 
\frac{2}{2^{\ast }}\right) ^{i+1}}{1-\frac{2}{2^{\ast }}}\right] =2\left( 
\frac{2^{\ast }}{2}\right) ^{i+1}\left[ 1-\beta \frac{2^{\ast }}{2^{\ast }-2}%
\right] +2\beta \frac{2^{\ast }}{2^{\ast }-2}
\end{split}
\label{ipassaggio}
\end{equation}%
and where 
\begin{equation*}
\begin{split}
c_{2}& =\,\prod_{k=0}^{+\infty }\left[ \frac{c_{1}}{(\bar{R}-\bar{\rho})^{2}}%
\,\delta _{k}^{2}\,2^{k+1}\right] ^{\left( \frac{2}{2^{\ast }}\right)
^{k}}\leq \prod_{k=0}^{+\infty }\left[ \frac{c_{1}}{(\bar{R}-\bar{\rho})^{2}}%
\,4\left( \frac{2^{\ast }}{2}\right) ^{2k}\,2^{k+1}\right] ^{\left( \frac{2}{%
2^{\ast }}\right) ^{k}} \\
& \leq \left( \frac{8c_{1}}{(\bar{R}-\bar{\rho})^{2}}\right)
^{\sum_{k=0}^{+\infty }\left( \frac{2}{2^{\ast }}\right) ^{k}}\,(2^{\ast
})^{\sum_{k=0}^{\infty }2k\left( \frac{2}{2^{\ast }}\right) ^{k}}=:\frac{%
c_{3}}{(\bar{R}-\bar{\rho})^{\frac{2\,2^{\ast }}{2^{\ast }-2}}}.
\end{split}%
\end{equation*}%
\newline
Now, by \eqref{H}$_{2}$ we have that, 
\begin{equation}
1+t^{2}g_{2}(t)\leq \frac{t}{g_{2}(1)}g_{2}(1)+t^{2}g_{2}(t)\leq \left( 
\frac{1}{g_{2}(1)}+1\right) t^{2}g_{2}(t).  \label{+1}
\end{equation}%
So we can write 
\begin{equation*}
\begin{split}
& \left[ \int_{B_{\bar{\rho}}}(1+(|Du|-1)_{+})^{\delta
_{i+1}}g_{2}(1+(|Du|-1)_{+})\,dx\right] ^{\left( \frac{2}{2^{\ast }}\right)
^{i+1}} \\
& \qquad \leq \,{\frac{c_{4}}{(\bar{R}-\bar{\rho})^{\frac{2\,2^{\ast }}{%
2^{\ast }-2}}}\int_{B_{\bar{R}}}\left[
(1+(|Du|-1)_{+})^{2}g_{2}(1+(|Du|-1)_{+})\right] \,dx,}
\end{split}%
\end{equation*}%
{$c_{4}=c_{3}\left( \frac{1}{g_{2}(1)}+1\right) $}. Finally, by %
\eqref{ipassaggio}, $\delta _{i+1}\left( \frac{2}{2^{\ast }}\right)
^{i+1}\rightarrow \left[ 2-\beta \frac{2\,2^{\ast }}{2^{\ast }-2}\right] $
as $i\rightarrow +\infty $, so passing to the limit we obtain 
\begin{equation*}
\begin{split}
\left( \Vert 1+(|Du|-1)_{+}\Vert _{L^{\infty }(B_{\rho })}\right) & ^{2-%
\frac{22^{\ast }}{2^{\ast }-2}\beta } \\
& =\lim_{i\rightarrow +\infty }\left[ \int_{B_{\bar{\rho}%
}}(1+(|Du|-1)_{+})^{\delta _{i+1}}g_{2}(1+(|Du|-1)_{+})\,dx\right] ^{\left( 
\frac{2}{2^{\ast }}\right) ^{i+1}} \\
& \leq \,{\ \frac{c_{5}}{(\bar{R}-\bar{\rho})^{\frac{2\,2^{\ast }}{2^{\ast
}-2}}}\int_{B_{\bar{R}}}(1+(|Du|-1)_{+})^{2}g_{2}(1+(|Du|-1)_{+})\,dx.}
\end{split}%
\end{equation*}%
Therefore \eqref{eq:normainfinto} is proved, in fact $n=\frac{2\,2^{\ast }}{%
2^{\ast }-2}$ if $n>2$ and $2<\frac{2\,2^{\ast }}{2^{\ast }-2}$ if $n=2$.
Notice that $0<2-n\beta <1$ since $\frac{1}{n}<\beta <\frac{2}{n}$.
\end{proof}



\begin{lemma}
\label{step 2} Assume that $f$ satisfies the assumptions of previous lemma
and \eqref{ab} and $u\in W_{\mathrm{loc}}^{1,\infty }(\Omega )$ is a local
minimizer of \eqref{energy integral}. Then for every $0<\rho <R$, $\bar{B}%
_{R}\subset\Omega $, there exists a positive constant $C=C(\rho,\, R,\,
C_{1},\, C_{2},\, \alpha,\, \beta,\, \mu,\, g_{2}(t_{0}))$, such that 
\begin{equation}
\Vert Du\Vert _{L^{\infty }(B_{\rho })}\leq \,C\,\left\{\frac 1{(R-\rho)^n}
\int_{B_{R}}(1+f(Du))\,dx\right\} ^{\theta}  \label{eq:apriori}
\end{equation}
with $\theta = \frac{(2-\mu)\alpha}{2-\mu-\alpha(n\beta-\mu) }$.
\end{lemma}

\begin{proof}
Set 
\begin{equation*}
V = V(x) =(1 + (|Du|-1)_+)^2 g_2(1 + (|Du|-1)_+).
\end{equation*}
By \eqref{H}$_2$ we have 
\begin{equation*}
\|V\|_{L^{\infty }(B_{\rho })}\leq C_\mu\left(\|1+(|Du|-1)_{+}\Vert
_{L^{\infty }(B_{\rho })}\right)^{2-\mu}
\end{equation*}
so inequality \eqref{eq:normainfinto} becomes 
\begin{equation*}
\left(\Vert V\Vert _{L^{\infty}(B_{\rho })}\right) ^{\frac{2-n\beta}{2-\mu}
} \leq \,\frac{c_5}{({R}-{\rho})^{n}} \int_{B_{R}}V(x)\,dx.
\end{equation*}
\newline
Let $\alpha>1$ satisfies \eqref{ab} we can apply Lemma \ref{interpolation}
with $v=V$, $\vartheta =\frac{2-\mu}{2-n\beta }$ and $\lambda=\frac 1\alpha$%
. In fact we have 
\begin{equation}  \label{alpha beta}
\alpha(n\beta-\mu)<2-\mu \quad\iff\quad \frac{2-\mu}{2-n\beta }\left(1-\frac
1\alpha\right)<1 \quad\iff\quad \vartheta(1-\lambda)<1 .
\end{equation}
Therefore we deduce the existence of a constant $c_{6}$ such that, for every 
$\varrho <R$, the following estimate holds 
\begin{equation*}
\left\Vert V\right\Vert _{L^{\infty }\left( B_{\rho }\right) } ^{\frac{%
1-\vartheta \left( 1-\lambda \right) }{\vartheta }} =\left\Vert
V\right\Vert_{L^{\infty }\left( B_{\rho }\right) } ^{\frac{%
2-\mu-\alpha(n\beta -\mu) }{(2-\mu)\alpha}} \leq \frac{c_6}{\left( R-\rho
\right) ^{n}}\int_{B_{R}}\left\vert V\right\vert ^{\frac 1\alpha }\,dx.
\end{equation*}
Now, by \eqref{H}$_4$, if $|Du|\geq 1$, 
\begin{equation*}
V=(1 + (|Du|-1)_+)^2 g_2(1+(|Du|-1)_+) = |Du|^2g_2(|Du|) \le \, C_2 \, (1 +
f(Du) )^{\alpha},
\end{equation*}
otherwise 
\begin{equation*}
V=(1 + (|Du|-1)_+)^2 g_2(1+(|Du|-1)_+) = g_2(1) \le \, g_2(1) \, (1 + f(Du)
)^{\alpha}.
\end{equation*}
Therefore 
\begin{equation}  \label{ultima}
\Vert V\Vert _{L^{\infty }\left( B_{\rho }\right) } \le \,\left[\frac{ c_7}{%
(R-\rho)^n} \int_{B_{R_0}} (1 + f(Du) )\, dx \right ]^{\frac{(2-\mu)\alpha}{%
2-\mu-\alpha(n\beta -\mu) }}
\end{equation}
holds for $c_7:= {\max\{C_2\,,g_2(1)\}^{\frac 1{\alpha}}} c_6$. Finally,
since by \eqref{H}$_2$ $V\geq g_2(1)|Du|$, \eqref{eq:apriori} holds for $C={%
c_7}^{\frac{(2-\mu)\alpha}{2-\mu-\alpha(n\beta -\mu) }}/g_2(1)$ and $\theta=%
\frac{(2-\mu)\alpha}{2-\mu-\alpha(n\beta-\mu) }$.
\end{proof}

\medskip


\section{Proofs of the results of Section \protect\ref{Section Introduction}}

We use the following approximation Lemma.

\begin{lemma}
Assume that $f:\mathbb{R}^n\to[0\,,+\infty)$ be a convex function and $v \in
W^{1,1}_{\mathrm{loc}}(\Omega)$ such that $f(D v) \in L_{\mathrm{loc}%
}^1(\Omega)$. For $\Omega^{\prime}$ open set compactly contained in $\Omega$
and $\varphi_\varepsilon$ be smooth mollifiers with support in $%
B_\varepsilon(0)$ we define $v_{\varepsilon}=v \ast \varphi _{\varepsilon}
\in{\mathcal{C}}^\infty(\Omega^{\prime})$ i.e. 
\begin{equation}  \label{molli}
v_{\varepsilon}(x)=\int_{B_\varepsilon(0)}\varphi _\varepsilon(y)v(x-y)\,dy,
\qquad x\in \Omega^{\prime}.
\end{equation}
Then, for every open ball $B_\rho$ compactly contained in $\Omega^{\prime}$, 
\begin{equation}  \label{eq:lim}
\lim_{\varepsilon \rightarrow 0} \int_{B_{\rho}} f(Dv_{\varepsilon})\,dx =
\int_{B_{\rho}} f(Dv) \, dx .
\end{equation}
\end{lemma}

\begin{proof}
By Jensen's inequality 
\begin{equation*}
f(Dv_{\varepsilon }(x))\leq \,\int_{B_{\varepsilon }(0)}\rho _{\varepsilon
}(y)\,f(Dv(x-y))\,dy.
\end{equation*}%
By integrating over $B_{\rho }$ for $\varepsilon $ sufficiently small we
obtain 
\begin{equation*}
\begin{split}
\int_{B_{\rho }}f(Dv_{\varepsilon }(x))\,dx\leq & \int_{B_{\varepsilon
}(0)}\rho _{\varepsilon }(y)\int_{B_{\rho }}f(Dv(x-y))\,dx\,dy \\
\leq & \int_{B_{\varepsilon }(0)}\rho _{\varepsilon }(y)\,dy\int_{B_{\rho
+\varepsilon }}f(Dv(x))\,dx\leq \,\int_{B_{\rho +\varepsilon }}f(Dv(x))\,dx
\end{split}%
\end{equation*}%
and then 
\begin{equation*}
\limsup_{\varepsilon \rightarrow 0}\int_{B_{\rho }}f(Dv_{\varepsilon
}(x))\,dx\leq \int_{B_{\rho }}f(Dv(x))\,dx.
\end{equation*}%
By the other hand, $Dv_{\varepsilon }$ converges to $Dv$ in $L^{1}(B_{\rho
}\,,\mathbb{R}^{n})$, then the lower semicontinuity of the integral yields 
\begin{equation*}
\liminf_{\varepsilon \rightarrow 0}\int_{B_{\rho }}f(Dv_{\varepsilon
}(x))\,dx\geq \int_{B_{\rho }}f(Dv(x))\,dx
\end{equation*}%
proving \eqref{eq:lim}.
\end{proof}


\begin{proof}[Proof of Theorem \protect\ref{superlinear}]
Consider the functional \eqref{energy integral} with $f$ satisfying \eqref{H}%
. 
For every $k\in \mathbb{N}$, let us consider the sequence $f_{k}$ defined as
follows (see \cite{MS}): 
\begin{equation}
f_{k}(\xi )=f(\xi )(1-\phi (\xi ))+(f\phi )\ast \eta _{k}(\xi ),  \label{fk}
\end{equation}%
where $\eta _{k}$ are standard mollifiers and $\phi \in \mathcal{C}^{\infty
}(\mathbb{R}^{n})$, $0\leq \phi (\xi )\leq 1$ for every $\xi \in \mathbb{R}%
^{n}$, $\phi (\xi )=1$ if $|\xi |\leq t_{0}+1$ and $\phi (\xi )=0$ if $|\xi
|\geq t_{0}+2$. $f_{k}\in \mathcal{C}^{2}(\mathbb{R}^{n})$ and 
\begin{equation*}
f_{k}(\xi )=%
\begin{cases}
f(\xi ) & |\xi |\geq t_{0}+2 \\ 
(f\phi )\ast \eta _{k}(\xi ) & |\xi |\leq t_{0}+1.%
\end{cases}%
\end{equation*}%
\newline
Therefore the sequence $\{f_{k}\}_{k}$ converges to $f$ uniformly and we can
suppose $|f(\xi )-f_{k}(\xi )|\leq 1$ for every $\xi \in \mathbb{R}^{N}$ and 
$k\in \mathbb{N}$. Moreover, for sufficiently large $k$, $f_{k}$ is a convex
function and $D^{2}f_{k}(\xi )$ is positive defined for $|\xi |>t_{0}+1$.
Since $f_{k}(\xi )=f(\xi )$ for $|\xi |>t_{0}+2$, then \eqref{H} holds with $%
t_{0}+2$ instead of $t_{0}$.

Let $h\in \mathcal{C}^{2}([0,+\infty ))$ be the positive, increasing, convex
function defined by 
\begin{equation*}
h(t)=%
\begin{cases}
\frac{1}{8}(6t^{2}-t^{4}+3) & t\in \lbrack 0\,,1) \\ 
t & t\in \lbrack 1\,,+\infty ).%
\end{cases}%
\end{equation*}%
Observe that $h\in \mathcal{C}^{2}(\mathbb{R}^{2})$, $h^{\prime \prime }$ is
non negative and $h^{\prime \prime }>0$ in $[0\,,1)$. For $k\in \mathbb{N}$
denote 
\begin{equation*}
\tilde{f}_{k}(\xi )=f_{k}(\xi )+\frac{1}{k}\,h\left( \frac{|\xi |}{t_{0}+2}%
\right)
\end{equation*}%
and define the integral functional 
\begin{equation*}
F_{k}(v)=\int_{B_{R}}\tilde{f}_{k}(Dv)\,dx.
\end{equation*}%
Notice that $\tilde{f}_{k}\in \mathcal{C}^{2}(\mathbb{R}^{n})$ and that it
is uniformly convex on compact subsets of $\mathbb{R}^{n}$. Let $B_{R}$ be a
ball compactly contained in $\Omega $ and $u\in W_{\mathrm{loc}%
}^{1,1}(\Omega )$ be a local minimizer of the functional 
\eqref{energy
integral}. Since $u_{\varepsilon }\in {\mathcal{C}}^{2}(\overline{B}_{R})$,
then it verifies the bounded slope condition (see for example \cite%
{Hartman-Stampacchia} and \cite[Theorem 1.1 and Theorem 1.2]{giusti}) and
then $F_{k}$ has unique minimizer $v_{\varepsilon ,k}$ among Lipschitz
continuous functions in $B_{R}$ with boundary value $u_{\varepsilon }$ on $%
\partial B_{R}$. By \eqref{H}$_{1}$ and \cite[(3.3)]{MP}, 
\begin{equation*}
g_{1}(|\xi |)|\lambda |^{2}\leq \sum_{i,j}(\tilde{f}_{k})_{\xi _{i}\xi
_{j}}(\xi )\lambda _{i}\lambda _{j}\leq \left[ g_{2}(|\xi |)|+\frac{1}{%
k(t_{0}+2)}\frac{1}{|\xi |}\right] |\lambda |^{2},
\end{equation*}%
for every $\lambda ,\xi \in \mathbb{R}^{n}$, $|\xi |\geq t_{0}+2$. On the
other hand, since by \eqref{H}$_{2}$ $t\,g_{2}(t)\geq
(t_{0}+2)g_{2}(t_{0}+2) $ for $t\geq t_{0}+2$, 
\begin{equation*}
g_{2}(t)+\frac{1}{k(t_{0}+2)}\frac{1}{t}\leq g_{2}(t)+\frac{g_{2}(t)}{%
k(t_{0}+2)^{2}g_{2}(t_{0}+2)}\leq 2g_{2}(t)
\end{equation*}%
for $k$ sufficiently large. Therefore $\tilde{f}_{k}$ satisfies \eqref{H}
with $t_{0}+2$ instead of $t_{0}$, $2g_{2}$ instead of $g_{2}$ and constants 
$C_{1}$ and $C_{2}$ independent from $k$. Moreover, as $k\rightarrow \infty $%
, $\tilde{f}_{k}\rightarrow f$ uniformly on compact subsets of $\mathbb{R}%
^{n}$ and then 
\begin{equation*}
\lim_{k\rightarrow \infty }\int_{B_{R}}\tilde{f}_{k}(Dv)\,dx=%
\int_{B_{R}}f(Dv)\,dx\qquad \text{ for every }v\in W^{1,\infty }(B_{R}).
\end{equation*}%
\newline
Since $\tilde{f}_{k}$ is uniformly convex on compact sets and $%
v_{\varepsilon ,k}$ is Lipschitz continuous in $B_{R}$, all the assumptions
of Lemma \ref{step 2} hold and, by the a priori estimate \eqref{eq:apriori},
for every $\rho <R$ there exists a positive constant $C$ depending on $\rho $%
, $R$, $\alpha $, $\beta $, $\mu $, $C_{1}$, $C_{2}$, $g_{2}(t_{0})$, but
independent on $k$ such that 
\begin{equation*}
\begin{split}
\Vert Dv_{\varepsilon ,k}\Vert _{L^{\infty }(B_{\rho })}& \leq \,C\left\{ 
\frac{1}{(R-\rho )^{n}}\int_{B_{R}}1+\tilde{f}_{k}(Dv_{\varepsilon
,k})\,dx\right\} ^{\theta } \\
& \leq \,C\left\{ \frac{1}{(R-\rho )^{n}}\int_{B_{R}}1+\tilde{f}%
_{k}(Du_{\varepsilon })\,dx\right\} ^{\theta },
\end{split}%
\end{equation*}%
$\theta =\frac{(2-\mu )\alpha }{2-\mu -\alpha (n\beta -\mu )}$, where the
last inequality depends by the minimality of $v_{\varepsilon ,k}$.
Therefore, for every $\varepsilon >0$ we have 
\begin{equation*}
\limsup_{k\rightarrow \infty }\Vert Dv_{\varepsilon ,k}\Vert _{L^{\infty
}(B_{\rho })}\leq \,C\,\left\{ \frac{1}{(R-\rho )^{n}}\int_{B_{R}}(1+f(Du_{%
\varepsilon })\,dx\right\} ^{\theta }=M_{\varepsilon }.
\end{equation*}%
The sequence $v_{\varepsilon ,k}$ is bounded in $W^{1,\infty }(B_{\rho })$
with respect to $k$, then there exists a subsequence $k_{j}\rightarrow 0$,
such that $\{v_{\varepsilon ,k_{j}}\}$ is weakly$^{\ast }$ convergent to $%
\bar{v}_{\varepsilon }$ in $W^{1,\infty }(B_{\rho })$ and for every $\rho <R$
\begin{equation}
\Vert D\bar{v}_{\varepsilon }\Vert _{L^{\infty }(B_{\rho })}\leq
\,C\,\left\{ \frac{1}{(R-\rho )^{n}}\int_{B_{R}}(1+f(Du_{\varepsilon
})\,dx\right\} ^{\theta }.  \label{normepsilon}
\end{equation}%
We prove that $v_{\varepsilon ,k_{j}}$ converges to $\bar{v}_{\varepsilon }$
in $W^{1,1}(B_{R})$ and then $\bar{v}_{\varepsilon }\in u+W_{0}^{1,1}(B_{R})$%
. Indeed by \eqref{H}$_{5}$ and the minimality of $v_{\varepsilon ,k_{j}}$,
as $j\rightarrow \infty $ we have 
\begin{equation*}
\int_{B_{R}}f(Dv_{\varepsilon ,k_{j}})\,dx\leq 1+\int_{B_{R}}\tilde{f}%
_{k_{j}}(Dv_{\varepsilon ,k_{j}})\,dx\leq 1+\int_{B_{R}}\tilde{f}%
_{k_{j}}(Du_{\varepsilon })\,dx\rightarrow 1+\int_{B_{R}}f(Du_{\varepsilon
})\,dx.
\end{equation*}%
By de la Vall\'{e}e-Poussin Theorem we can choose the sequence $k_{j}$ such
that $Dv_{\varepsilon ,k_{j}}\rightharpoonup D\bar{v}_{\varepsilon }$ in $%
L^{1}(B_{R})$ and then $v_{\varepsilon ,k_{j}}-u_{\varepsilon
}\rightharpoonup (\bar{v}_{\varepsilon }-u_{\varepsilon })\in
W_{0}^{1,1}(B_{R})$. On the other hand, for every $\delta >0$ and for every $%
k$ sufficiently large, $|\tilde{f}_{k}(\xi )-f(\xi )|\leq \delta $, for
every $|\xi |\leq M+1$, therefore by the minimality of $v_{\varepsilon
,k_{j}}$ 
\begin{equation*}
\begin{split}
\int_{B_{\rho }}f(Dv_{\varepsilon ,k_{j}})\,dx& =\int_{B_{\rho
}}(f(Dv_{\varepsilon ,k_{j}})-\tilde{f}_{k_{j}}(Dv_{\varepsilon ,k_{j}}))+%
\tilde{f}_{k_{j}}(Dv_{\varepsilon ,k_{j}})\,dx \\
& \leq \int_{B_{R}}\tilde{f}_{k_{j}}(Du_{\varepsilon })\,dx+\delta |B_{R}|.
\end{split}%
\end{equation*}%
By lower semicontinuity in $W^{1,1}(B_{R})$, passing to the limit for $%
j\rightarrow \infty $, we get 
\begin{equation*}
\begin{split}
\int_{B_{\rho }}f(D\bar{v}_{\varepsilon })\,dx& \leq \liminf_{j\rightarrow
\infty }\int_{B_{\rho }}f(Dv_{\varepsilon ,k_{j}})\,dx\leq
\lim_{j\rightarrow \infty }\int_{B_{R}}\tilde{f}_{k_{j}}(Du_{\varepsilon
})\,dx+\delta |B_{R}| \\
& =\int_{B_{R}}f(Du_{\varepsilon })\,dx+\delta |B_{R}|
\end{split}%
\end{equation*}%
for every $\delta >0$ and $\rho <R$, and then for $\rho \rightarrow R$ and $%
\delta \rightarrow 0$ 
\begin{equation*}
\int_{B_{R}}f(D\bar{v}_{\varepsilon })\,dx\leq \int_{B_{R}}f(Du_{\varepsilon
})\,dx
\end{equation*}%
Again by de la Vall\'{e}e-Poussin Theorem and \eqref{normepsilon}, we have
that there exists a sequence $\varepsilon _{j}\rightarrow 0$ such that $\bar{%
v}_{\varepsilon _{j}}-u_{\varepsilon _{j}}\rightharpoonup \bar{v}-u$ in $%
W_{0}^{1,1}(B_{R})$ and $\{\bar{v}_{\varepsilon _{j}}\}_{j}$ is weakly $%
^{\ast }$ convergent to $\bar{v}$ in $W^{1,\infty }(B_{\rho })$ for every $%
0<\rho <R$. By the lower semicontinuity of the functional 
\begin{equation}
\int_{B_{R}}f(D\bar{v})\,dx\leq \liminf_{j\rightarrow \infty }\int_{B_{R}}f(D%
\bar{v}_{\varepsilon _{j}})\,dx\leq \lim_{j\rightarrow \infty
}\int_{B_{R}}f(Du_{\epsilon _{j}})\,dx=\int_{B_{R}}f(Du)\,dx.  \label{eq:min}
\end{equation}%
Then $\bar{v}$ is a minimizer for \eqref{energy integral} with $\Omega
=B_{R} $. Moreover from \eqref{eq:lim} and \eqref{normepsilon} we have that $%
\bar{v}_{\varepsilon }$ converges to $\bar{v}$ in $W_{\mathrm{loc}%
}^{1,\infty }$ and 
\begin{equation}
\begin{split}
\Vert D\bar{v}\Vert _{L^{\infty }(B_{\rho })}& \leq \,\liminf_{j\rightarrow
\infty }\Vert D\bar{v}_{\varepsilon _{j}}\Vert _{L^{\infty }(B_{\rho })}\leq
\lim_{j\rightarrow \infty }\,C\,\left\{ \frac{1}{(R-\rho )^{n}}%
\int_{B_{R}}(1+f(Du_{\varepsilon _{j}})\,dx\right\} ^{\theta } \\
& =\,C\,\left\{ \frac{1}{(R-\rho )^{n}}\int_{B_{R}}(1+f(Du)\,dx\right\}
^{\theta }.
\end{split}
\label{eq:lips}
\end{equation}%
Therefore $\bar{v}$ and $u$ are two different minimizers of $F$ in $B_{R}$.
Since $f(\xi )$ is strictly convex for $|\xi |>t_{0}$, by proceeding as in 
\cite{EMM1} it is possible to prove that the set 
\begin{equation*}
E_{0}:=\left\{ x\in B_{R}:\left\vert \frac{Du(x)+D\bar{v}(x)}{2}\right\vert
>t_{0}\right\} .
\end{equation*}%
has zero measure. Therefore 
\begin{equation*}
\Vert Du\Vert _{L^{\infty }(B_{\rho })}\leq \,\Vert Du+D\bar{v}\Vert
_{L^{\infty }(B_{\rho })}+\Vert D\bar{v}\Vert _{L^{\infty }(B_{\rho })}\leq
\,2t_{o}+\Vert D\bar{v}\Vert _{L^{\infty }(B_{\rho })}.
\end{equation*}
\end{proof}



\begin{lemma}
\label{p<f<q} Let $f$ be a positive function in $\mathcal{C}^2(\{\xi\in%
\mathbb{R}^n:\,|\xi|\geq t_0\})$ satisfying ellipticity conditions %
\eqref{ellipticity conditions under p,q}. Then there exists $\bar t \geq t_0$
such that 
\begin{equation}  \label{pfq}
\frac m{2(p-1)}|\xi|^p \leq f(\xi) \leq \frac {2M}{q-1}|\xi|^q, \qquad
|\xi|\geq\bar t
\end{equation}
if $1<p\leq q$, $f(\xi)\geq \frac m2 |\xi|\log |\xi|$ or $f(\xi)\leq 2M
|\xi|\log |\xi|$ if $p=1$ or $q=1$ respectively.
\end{lemma}

\begin{proof}
Without loss of generality we can suppose $t_{0}=1$. We will prove the left
hand side in \eqref{pfq}, the right hand side being analogous. We consider
the real function $\varphi :\left[ 1\,,+\infty \right) \rightarrow \mathbb{R}
$ defined by $\varphi \left( t\right) =f\left( t\,\frac{\xi }{|\xi |}\right) 
$. Then $\varphi \in \mathcal{C}^{2}\left[ 1\,,+\infty \right) $ and its
first and second derivatives hold 
\begin{equation*}
\varphi ^{\prime }\left( t\right) =\left( D_{\xi }f\left( t\,\frac{\xi }{%
|\xi |}\right) ,\frac{\xi }{|\xi |}\right) \,;\;\;\;\;\;\;\varphi ^{\prime
\prime }\left( t\right) =\sum_{i,j=1}^{n}f_{\xi _{i}\xi _{j}}\left( t\,\frac{%
\xi }{|\xi |}\right) \frac{\xi _{i}\xi _{j}}{|\xi |^{2}}
\end{equation*}%
and, by \eqref{ellipticity conditions under p,q}, $\varphi ^{\prime \prime
}(t)\geq mt^{p-2}$. Therefore, integrating from $1$ to $t$ we obtain 
\begin{equation*}
\varphi ^{\prime }(t)-\varphi ^{\prime }(1)\geq 
\begin{cases}
\frac{m}{p-1}\left( t^{p-1}-1\right) & \text{ if }p>1 \\ 
m\log t & \text{ if }p=1%
\end{cases}%
\end{equation*}%
and again 
\begin{equation*}
\varphi (t)-\varphi (1)\geq Df\left( \frac{\xi }{|\xi |}\right) (t-1)+%
\begin{cases}
\frac{m}{p-1}\left( t^{p}-t\right) & \text{ if }p>1 \\ 
m(t\log t-t) & \text{ if }p=1.%
\end{cases}%
\end{equation*}%
Therefore, for $t=|\xi |$ we have 
\begin{equation*}
f(\xi )\geq -L(|\xi |-1)+%
\begin{cases}
\frac{m}{p-1}\left( |\xi |^{p}-|\xi |\right) & \text{ if }p>1 \\ 
m(|\xi |\log |\xi |-|\xi |) & \text{ if }p=1%
\end{cases}%
\end{equation*}%
where $L=\max \{|Df(v)|:\,|v|=1\}$.
\end{proof}


\bigskip

\begin{proof}[Proof of Corollary~\protect\ref{theorem on p,q}]
We have to prove that the assumptions of Theorem \ref{superlinear} are
satisfied. First of all, notice that assumptions \eqref{H}$_{1}$ and\eqref{H}%
$_{2}$ hold for $g_{1}(t)=mt^{p-2}$ and $g_{2}(t)=Mt^{q-2}$. Moreover, for
every $t\geq 1$, 
\begin{equation*}
(g_{2}(t))^{\frac{2}{2^{\ast }}}=M^{{2}/{2^{\ast }}}t^{(q-2)\frac{2}{2^{\ast
}}}=mt^{p-2}\frac{M^{{2}/{2^{\ast }}}}{m}t^{(q-2)\frac{2}{2^{\ast }}-(p-2)},
\end{equation*}%
then \eqref{H}$_{3}$ holds for $\beta \geq \bar{\beta}=\frac{n-2}{2n}q-\frac{%
p}{2}+\frac{2}{n}$. By Lemma~\ref{p<f<q}, for $\alpha \geq \bar{\alpha}=q/p$
and $|\xi |$ large enough, 
\begin{equation*}
g_{2}(|\xi |)|\xi |^{2}=M|\xi |^{q}\leq M\left( \frac{2p-2}{m}\right)
^{\alpha }\left( \frac{m}{2p-2}|\xi |^{p}\right) ^{\alpha }\leq M\left( 
\frac{4}{m}\right) ^{\alpha }[1+f(\xi )]^{\alpha }
\end{equation*}%
and \eqref{H}$_{4}$ holds. Moreover, by Lemma~\ref{p<f<q} also \eqref{H}$%
_{5} $ holds. Therefore, all the assumptions of Theorem \ref{superlinear}
are satisfied if \eqref{ab} is satisfied, that is 
\begin{equation*}
\alpha <\frac{q}{n\beta +q-2}\quad \iff \quad \frac{q}{p}<\frac{q}{\frac{n}{2%
}q-\frac{n}{2}p}\quad \iff \quad \frac{q}{p}<\frac{n+2}{n}
\end{equation*}%
since $\mu =2-q$. In order to obtain the correct exponent in 
\eqref{bound for
p,q<=2}, notice that in this case by \eqref{ultima} we can conclude 
\begin{equation*}
\Vert Du\Vert _{L^{\infty }\left( B_{\rho }\right) }^{q}\leq \,\left[ \frac{%
c_{7}}{(R-\rho )^{n}}\int_{B_{R_{0}}}(1+f(Du))\,dx\right] ^{\frac{q\alpha }{%
q-\alpha (n\beta +q-2)}}.
\end{equation*}%
Since $\frac{q\alpha }{q-\alpha (n\beta +q-2)}=\frac{2}{(n+2)p-nq}$, we get %
\eqref{bound for p,q<=2}.
\end{proof}


\end{document}